\documentclass[12pt,oneside,english]{amsart}
\usepackage[latin9]{inputenc}
\usepackage{color}
\usepackage{babel}
\usepackage{amsthm}
\usepackage{amssymb}
\usepackage{graphicx}
\usepackage{setspace}
\usepackage[authoryear]{natbib}
\usepackage[unicode=true,pdfusetitle,
 bookmarks=true,bookmarksnumbered=false,bookmarksopen=false,
 breaklinks=true,pdfborder={0 0 0},pdfborderstyle={},backref=false,colorlinks=true]
 {hyperref}
\hypersetup{
 linkcolor=[rgb]{0,0,0.5}, citecolor=[rgb]{0,0.3,0}}

\makeatletter
\theoremstyle{plain}
\newtheorem{thm}{\protect\theoremname}
  \theoremstyle{definition}
  \newtheorem{defn}[thm]{\protect\definitionname}
  \theoremstyle{definition}
  \newtheorem{condition}[thm]{\protect\conditionname}
  \theoremstyle{plain}
  \newtheorem{lem}[thm]{\protect\lemmaname}
  \theoremstyle{plain}
  \newtheorem{prop}[thm]{\protect\propositionname}
  \theoremstyle{remark}
  \newtheorem{rem}[thm]{\protect\remarkname}
  \theoremstyle{plain}
  \newtheorem{cor}[thm]{\protect\corollaryname}

\usepackage{fullpage,setspace,mathbbol}

\DeclareMathOperator*{\argsup}{arg\,sup}

\bibpunct[; ]{(}{)}{,}{a}{}{;}

\makeatother

  \providecommand{\conditionname}{Condition}
  \providecommand{\corollaryname}{Corollary}
  \providecommand{\definitionname}{Definition}
  \providecommand{\lemmaname}{Lemma}
  \providecommand{\propositionname}{Proposition}
  \providecommand{\remarkname}{Remark}
\providecommand{\theoremname}{Theorem}

\begin{document}

\title{An inverse optimal stopping problem for diffusion processes\\
\bigskip{}
\today}

\author{Thomas Kruse and Philipp Strack}
\begin{abstract}
Let $X$ be a one-dimensional diffusion and let $g\colon[0,T]\times\mathbb{R}\to\mathbb{R}$
be a payoff function depending on time and the value of $X$. The
paper analyzes the inverse optimal stopping problem of finding a time-dependent
function $\pi:[0,T]\to\mathbb{R}$ such that a given stopping time
$\tau^{\star}$ is a solution of the stopping problem $\sup_{\tau}\mathbb{E}\left[g(\tau,X_{\tau})+\pi(\tau)\right]\,.$

Under regularity and monotonicity conditions, there exists a solution
$\pi$ if and only if $\tau^{\star}$ is the first time when $X$
exceeds a time-dependent barrier $b$, i.e. $\tau^{\star}=\inf\left\{ t\ge0\,|\,X_{t}\ge b(t)\right\} \,.$
We prove uniqueness of the solution $\pi$ and derive a closed form
representation. The representation is based on an auxiliary process
which is a version of the original diffusion $X$ reflected at $b$
towards the continuation region. The results lead to a new integral
equation characterizing the stopping boundary $b$ of the stopping
problem $\sup_{\tau}\mathbb{E}\left[g(\tau,X_{\tau})\right]$. \medskip{}

 \medskip{}

\end{abstract}

\thanks{{\footnotesize{}\qquad{}Thomas Kruse, University of Duisburg-Essen,
Thea-Leymann-Str. 9, 45127 Essen, Germany. Email: thomas.kruse@uni-due.de;}}

\thanks{{\footnotesize{}\qquad{}Philipp Strack, University of Califonia,
Berkeley, 513 Evans Hall, Berkeley, 94720 California, USA. Web: \href{http://www.philippstrack.com}{http://www.philippstrack.com},
Email: philipp.strack@gmail.com. }}

\thanks{{\footnotesize{}\qquad{}We would like to thank Stefan Ankirchner,
Dirk Bergemann, Paul Heidhues, David Hobson, Monique Jeanblanc, Daniel
Krähmer, Benny Moldovanu, Jenö Pál, Goran Peskir, Sven Rady, Albert
Shiryaev, Juuso Toikka and seminar audiences in Austin, Bonn, HU Berlin,
Berkeley, Bielefeld, Boston, Caltech, Columbia, Duke, Harvard, Maastricht,
Microsoft Research NE, MIT theory camp, Northwestern, NYU, Paris Evry,
Toronto, University College London, UPenn, Warwick, Yale for many
helpful comments and suggestions. Any remaining errors are our own.
Financial support by the Bonn Graduate School of Economics, the Bonn
International Graduate School of Mathematics, the German Research
Foundation (DFG), the Hausdorff Center for Mathematics, the SFB TR
15 and the French Banking Federation through the Chaire ``Markets
in Transition'' is gratefully acknowledged.}}
\maketitle

\section*{Introduction}

Optimal stopping is omnipresent in applications of dynamic optimization
in economics, statistics, and finance. Examples are the optimal exercise
timing of options, when to stop searching, and the quickest detection
problem. One method to solve optimal stopping problems in a Markovian
framework is to identify the stopping region. Optimal stopping times
are then given as first hitting times of the stopping region.

Many applications of optimal stopping naturally lead to the question
of how to change a payoff such that a given stopping rule becomes
optimal. Mathematically, this \textit{inverse} optimal stopping problem
consists of modifying the payoff of a stopping problem in such a way
that it is optimal to stop at the first time when a \textit{given}
set is hit. In many economic applications informational constraints
furthermore restrict the set of admissible modifications to the addition
of a time-dependent function to the original payoff, i.e. transfers
which are independent of the realization of the process.

To fix ideas, consider the continuous-time, finite horizon optimal
stopping problem 
\[
\sup_{\tau\in\mathcal{T}}\mathbb{E}\left[g(\tau,X_{\tau})\right]\,,
\]
where $\mathcal{T}$ is the set of stopping times with values in $[0,T]$,
$X$ is a one-dimensional diffusion and $g$ is a smooth payoff function.
A deterministic function $\pi:[0,T]\to\mathbb{R}$ is called a transfer.
We say that a set $A\subset[0,T]\times\mathbb{R}$ is implemented
by a transfer $\pi$ if the first time $\tau_{A}$ when $X$ hits
$A$ is optimal in the stopping problem with payoff $g+\pi,$ i.e.
if
\begin{equation}
\tau_{A}\in\argsup_{\tau\in\mathcal{T}}\mathbb{E}\left[g(\tau,X_{\tau})+\pi(\tau)\right].\label{eq:impl_def_intro}
\end{equation}

Inverse optimal stopping problems play an important role in different
economic situations. One example are dynamic principal-agent models:
There is an agent who \textit{privately} observes the stochastic process
$X$ and aims at maximizing her expected payoff $\sup_{\tau\in\mathcal{T}}\mathbb{E}\left[g(\tau,X_{\tau})\right]$
from stopping the process. The principal observes the stopping decision
of the agent, but not the realization of the process. She aims at
inducing the agent to take a particular stopping decision given by
the hitting time $\tau_{A}$. In order to influence the agent's stopping
decision the principal commits to a transfer $\pi$ \textendash{}
a payment which is due at the moment when the agent stops. The principal
needs to construct the transfer $\pi$ in such a way that $\tau_{A}$
becomes optimal in the modified stopping problem $\sup_{\tau\in\mathcal{T}}\mathbb{E}\left[g(\tau,X_{\tau})+\pi(\tau)\right]$.
For example, the agent could be a firm that has developed a new technology
and now has to decide when to introduce it to the market place. The
firm observes private signals regarding the demand, and this knowledge
changes over time. The principal is a social planner who also takes
the consumer surplus of the new technology into account and hence
prefers a different stopping decision than the firm. The inverse optimal
stopping problem analyzes the question how the planner can align the
preferences of the firm by subsidizing the market entry through a
transfer (see Subsection \ref{subsec:Providing-Incentives} for a
specific example).

Other economic examples of inverse optimal stopping problems are the
design of unemployment benefits \citet{mccall1970economics,hopenhayn1997optimal},
the structuring of management compensation, the sale of irreversible
investment options \citet{board2007selling}, as well as well as the
inference of deliberation costs in search theory \citet{drugowitsch2012cost,fudenberg2015stochastic}.
Section \ref{sec:Two-Motivating-Examples} presents two more specific
examples. For further economic examples and applications to revenue
management we refer to \citet{krusestrack2013optimal}, where inverse
optimal stopping problems have been introduced in a discrete-time
framework. 

The main result (Theorem \ref{thm:Cut-offs_implementable}) states
that all cut-off regions $A=\left\{ (t,x)\,|\,x\ge b(t)\right\} $
are implementable provided that the boundary $b$ is càdlàg and has
summable downwards jumps. Moreover, we suppose that a so-called single
crossing condition is satisfied. It requires that the expected gain
of waiting an infinitesimal amount of time is non-increasing in the
value of the process $X.$ Formally, we suppose that the function
\begin{equation}
x\mapsto\lim_{h\searrow0}\frac{1}{h}\mathbb{E}\left[g(t+h,X_{t+h}^{t,x})-g(t,x)\right]=(\partial_{t}+\mathcal{L})g(t,x)\label{eq:single_crossing_intro}
\end{equation}
is non-increasing, where $\mathcal{L}$ denotes the generator of the
diffusion $X$. Furthermore, we show that the solution $\pi$ implementing
the cut-off region $A=\left\{ (t,x)|x\ge b(t)\right\} $ admits the
following closed form representation 

\begin{equation}
\pi(t)=\mathbb{E}\left[\int_{t}^{T}(\partial_{t}+\mathcal{L})g(s,\tilde{X}_{s}^{t,b(t)})\mathrm{d}s\right]\,.\label{eq:transfer_intro}
\end{equation}

Here $(\tilde{X}_{s}^{t,b(t)})_{s\geq t}$ denotes the unique process
starting on the barrier $b(t)$ at time $t$ which results from reflecting
the original process $X$ at the barrier $(b(s))_{s\in[t,T]}$ away
from $A$. 

As shown in \citet{kotlow1973}, \citet{jacka1992finite} and \citet{villeneuve2007threshold}
the single crossing condition (or a weaker version of it) ensures
that the stopping region in stopping problems of the form $v(t,x)=\sup_{\tau\in\mathcal{T}_{t,T}}\mathbb{E}\left[g(\tau,X_{\tau}^{t,x})\right]$
is of cut-off type, i.e. there exists a barrier $b:[0,T]\to\mathbb{R}$
such that $x\ge b(t)$ if and only if $v(t,x)=g(t,x)$. In Proposition
\ref{prop:only_cut-offs_implementable} we show that this result translates
to implementable regions. We introduce the notion of \textit{strict}
implementability for sets $A\subset[0,T]\times\mathbb{R}$, where
we additionally demand that $A$ coincides with the stopping region
of the problem (\ref{eq:impl_def_intro}). Proposition \ref{prop:only_cut-offs_implementable}
states that under the single crossing condition only cut-off regions
are strictly implementable. Furthermore, we show that if the monotonicity
in Equation (\ref{eq:single_crossing_intro}) is strict, then cut-off
regions with a càdlàg barrier with summable downward jumps are strictly
implementable (Corollary \ref{thm:strict_single_crossing_impl_strong_impl}).
In this way the following characterization of strictly implementable
regions holds up the assumption of right continuity and summable downward
jumps: A region is strictly implementable if and only if it is of
cut-off type.

Furthermore, the transfer implementing a cut-off region is unique
up to an additive constant (Theorem \ref{thm:Uniqueness}). This result
leads to a new characterization of optimal stopping boundaries (Corollary
\ref{cor:appl_opt_stopping}). If the first hitting time $\tau_{A}$
of a set $A$ is optimal in the stopping problem $\sup_{\tau\in\mathcal{T}}\mathbb{E}\left[g(\tau,X_{\tau})\right]$
then $A$ is implemented by the zero transfer. Uniqueness of the transfer
implies that 
\begin{equation}
\mathbb{E}\left[\int_{t}^{T}(\partial_{t}+\mathcal{L})g(s,\tilde{X}_{s}^{t,b(t)})\mathrm{d}s\right]=0\label{eq:integral_equation_intro}
\end{equation}
for all $t\in[0,T].$ Remarkably, the nonlinear integral Equation
(\ref{eq:integral_equation_intro}) is not only necessary but also
sufficient for optimality.   In Section \ref{sec:Application-To-Optimal-Stopping}
we discuss the relation to the integral equation derived in \citet{kim1990analytic},
\citet{jacka1991optimal} and \citet{carr2006alternative} (see also
\citet{pevskir2006optimal}). 

The paper is organized as follows. Section \ref{sec:Two-Motivating-Examples}
presents specific examples of inverse optimal stopping problems. In
Section \ref{sec:Problem-Formulation} we set up the model and introduce
the notion of implementability. In Section \ref{sec:Implemtable_implies_cut-off}
we show that only cut-off regions are strictly implementable. Section
\ref{sec:Implementability-of-Cut-Off} is devoted to the converse
implication. First we introduce reflected processes and formally derive
the representation (\ref{eq:transfer_intro}) of the transfer (Subsection
\ref{subsec:Constrained-Processes}). Subsection \ref{subsec:Cut-offs_implementable}
contains the main results about implementability of cut-off regions.
In Subsection \ref{subsec:Properties-of-the-transfer} we present
the main properties of the transfer and in Subsection \ref{subsec:Uniqueness-of-The-transfer}
we provide the uniqueness result. In Section \ref{sec:Application-To-Optimal-Stopping}
we derive and discuss the integral equation (\ref{eq:integral_equation_intro}).

\section{Motivating Examples\label{sec:Two-Motivating-Examples}}

\subsection{\label{subsec:Providing-Incentives}Providing Incentives for Investment
in a Project of Unknown Profitability}

A single agent (or firm) can invest into a project of unknown value
$\theta\in\mathbb{R}$. The value (or discounted expected future
return) of the project $\theta\in\mathbb{R}$ is normally distributed
with mean $X_{0}$ and variance $\sigma_{0}^{2}$. The agent learns
about the project's value over time by observing a signal (or payoff)
$(Z_{t})$ which is a Brownian motion $(W_{t})$ (independent of $\theta$)
plus drift equal to the true return of the project
\[
\mathrm{d}Z_{t}=\theta\,\mathrm{d}t+\mathrm{d}W_{t}\,.
\]
When the agent invests into the project at time $\tau$ he receives
its value discounted by the time at which he invested $e^{-r\,\tau}\theta\,.$
The agent's problem is to find a stopping time adapted to $\mathbb{F}=\left(\mathcal{F}_{t}\right)_{t\ge0}$
(the natural filtration of $Z$) which solves $\sup_{\tau}\mathbb{E}\left[e^{-r\,\tau}\theta\right]\,.$
If we denote by $X_{t}=\mathbb{E}\left[\theta\mid\mathcal{F}_{t}\right]$
the posterior expected value that the agent assigns to the project,
the law of iterated expectations implies that this problem is equivalent
to
\[
\sup_{\tau}\mathbb{E}\left[e^{-r\,\tau}X_{\tau}\right]\,.
\]
It is well known (cf. Theorem 10.1 in \citet{liptser2013statistics})
that after seeing the signal $(Z_{s})_{s\leq t}$ the agents posterior
belief about the value of the project is normally distributed with
variance $\sigma_{t}^{2}=\frac{1}{\sigma_{0}^{-2}+t}$ and mean $X_{t}=\sigma_{t}^{2}(X_{0}\sigma_{0}^{-2}+Z_{t})\,,$
and furthermore that there exists a Brownian motion $B_{t}$ (in the
filtration $\mathbb{F}$) such that
\[
\mathrm{d}X_{t}=\sigma_{t}^{2}\,\mathrm{d}B_{t}\,.
\]
Hence, the agent's learning and investment problem is equivalent to
the problem of stopping the diffusion $X$. As the problem is Markovian
in $(t,X)$ and the returns from waiting to invest are decreasing
in $X_{t}$ it follows that the optimal solution is to invest once
the expected value of the project $X_{t}$ exceeds a time-dependent
threshold $b^{0}$\footnote{see also Proposition \ref{prop:only_cut-offs_implementable} below.}
\[
\tau=\inf\{t\colon X_{t}\geq b^{0}(t)\}\,.
\]
The decision to invest into a project might for example correspond
to bringing a new product to the market. In many such investment situations
the incentive of the firm to invest is not aligned with the incentives
of society. For example, a pharmaceutical firm which takes an investment
decision based upon the profitability of a treatment, ignoring consumer
surplus generated from the availability of medicine, will invest too
late and in too few treatments. To mitigate this inefficiency the
government could use subsidies (and taxes) on new projects which depend
on the time the firm invests and brings the project to the market.
For example, in Figure\textbf{ }\ref{fig:example} the dashed line
shows the investment threshold $b^{0}\colon[0,T]\to\mathbb{R}$ at
which the firm invests without a transfer. Suppose that the government
wants the firm to invest earlier, for example at the first time when
the firm's belief about the investment value $X$ exceeds the barrier
$b^{\pi}\colon[0,T]\to\mathbb{R},b^{\pi}(t)=0.5\sqrt{T-t}$. In Section
\ref{sec:Implementability-of-Cut-Off} below we show that this is
possible for the government by using a transfer $\pi\colon[0,T]\to\mathbb{R}$.
The solid line in the left-hand side of Figure \ref{fig:example}
depicts such a transfer.

\subsection{Quickest change point detection in a principal-agent framework\label{subsec:quickest_detect}}

Quickest detection problems play a prominent role in mathematical
statistics and are a key ingredient in a number of models in the applied
sciences such as quality control, epidemiology and geology (see, e.g.,
\citet[Chapter IV][]{shiryaev2007optimal}, \citet[Chapter IV, Section 22][]{peskir2006optimal}
and \citet{muller1994change} for historical accounts on the problem
formulations and specific applications). As an economic motivation
consider a venture capitalist and an entrepreneur. The entrepreneur
observes an informative signal about whether it is still profitable
to run the firm or not. This information is often unobservable to
the venture capitalist as he possesses no knowledge of the specific
market. The venture capitalist who finances the firm wants to stop
operations once he is $60\%$ sure that the firm became unprofitable.
The entrepreneur might prefer running the firm much longer as doing
so yields private benefits to him. An important question in the venture
capital industry is how to design compensation schemes which align
the interests of the entrepreneur with those of the venture capitalist. 

We consider here a variant of the quickest detection problem of a
Wiener process from a principal-agent perspective. For the formulation
of the single-agent quickest detection problem we follow closely \citet{gapeev2006wiener}.
There is an agent observing on the finite time interval $[0,T]$ the
path of a one-dimensional Brownian motion $X$ which changes its drift
from $0$ to $\mu\ne0$ at some random time $\theta$. The random
time $\theta$ is independent of $X$ and is exponentially distributed
with parameter $\lambda\in(0,\infty)$. The agent does not observe
$\theta$, but has to infer information about $\theta$ from the continuous
observation of $X$. The agent's goal is to find a stopping time of
$X$ that is as close to $\theta$ as possible. More formally, for
$c\in(0,\infty)$ the agent aims at finding a $[0,T]$-valued stopping
time $\tau$ with respect to the filtration $\mathcal{F}^{X}$ generated
by $X$ that attains the minimum in 
\[
\inf_{0\le\tau\le T}\left(\mathbb{P}\left[\tau\le\theta\right]+c\,\mathbb{E}\left[(\tau-\theta)^{+}\right]\right).
\]
As shown in \citet{gapeev2006wiener} this is equivalent to solving
the stopping problem
\begin{equation}
\inf_{0\le\tau\le T}\mathbb{E}\left[1-p_{\tau}+c\int_{0}^{\tau}p_{t}dt\right]=1-\sup_{0\le\tau\le T}\mathbb{E}\left[\int_{0}^{\tau}\lambda-(c+\lambda)p_{t}dt\right],\label{eq:quick_detect_stopping_prob}
\end{equation}
where the process $p$ satisfies for all $t\in[0,T]$ that
\[
dp_{t}=\lambda(1-p_{t})dt+\mu p_{t}(1-p_{t})dW_{t},\quad p_{0}=0
\]
and where $W$ is a standard Brownian motion. The process $p$ satisfies
for all $t\in[0,T]$ that $p_{t}=\mathbb{P}\left[\theta\le t\:|\:\mathcal{F}_{t}^{X}\right]$
and thus describes at each time $t\in[0,T]$ the posterior belief
whether $\theta$ already occurred. Now suppose that there is a principal
who does neither observe $X$ nor $\theta$, but is notified at the
moment when the agent stops. The principal's goal is to construct
a transfer $\pi\colon[0,T]\to\mathbb{R}$ to the agent such that the
principal is notified at the first time before $T$ when the posterior
belief $p$ exceeds a threshold level of, say, 60\%. It follows from
the results in Section \ref{sec:Implementability-of-Cut-Off} below
that this is possible (note in particular that the flow payoff $p\mapsto\lambda-(c+\lambda)p$
in (\ref{eq:quick_detect_stopping_prob}) is a decreasing function,
cf. Condition \ref{def:SingleCrossing}). The right-hand side of Figure
\ref{fig:example} shows such a transfer $\pi$.

\begin{figure}
\includegraphics[width=0.5\textwidth]{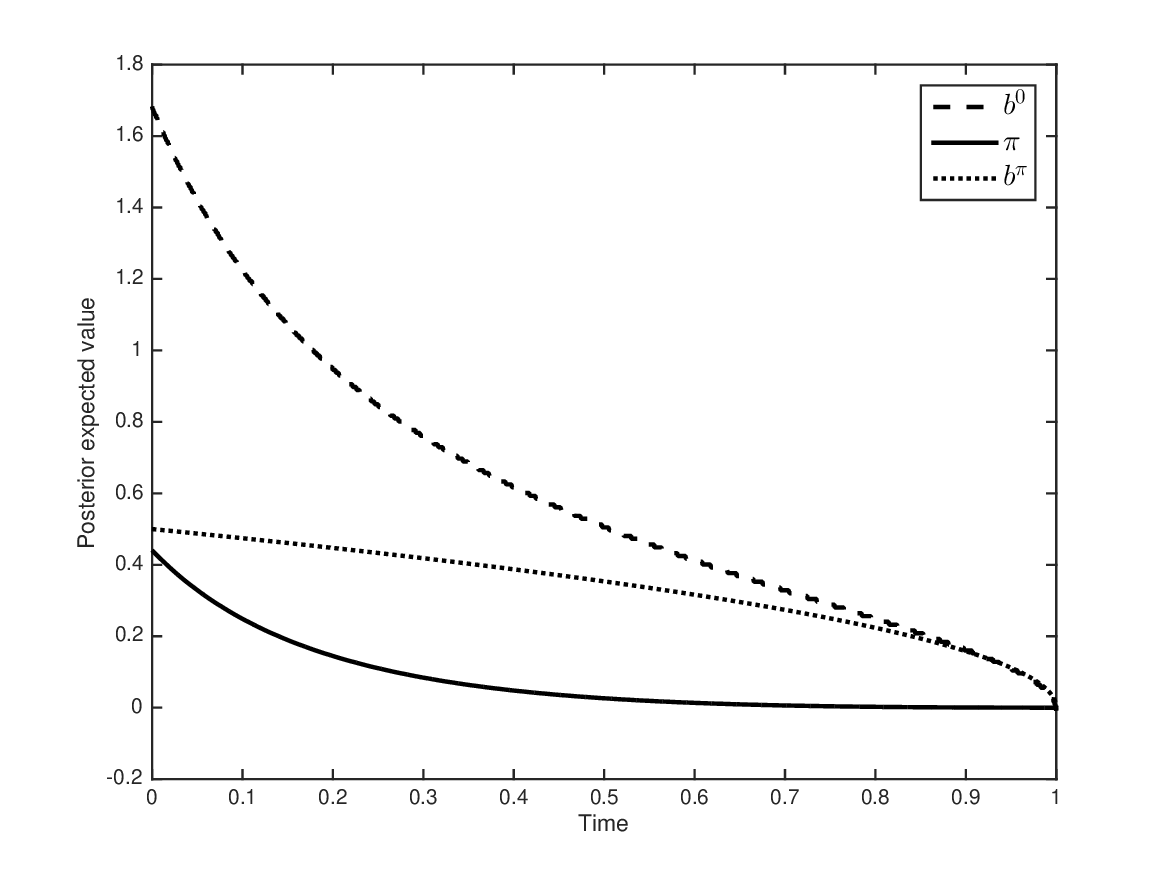}\includegraphics[width=0.5\textwidth]{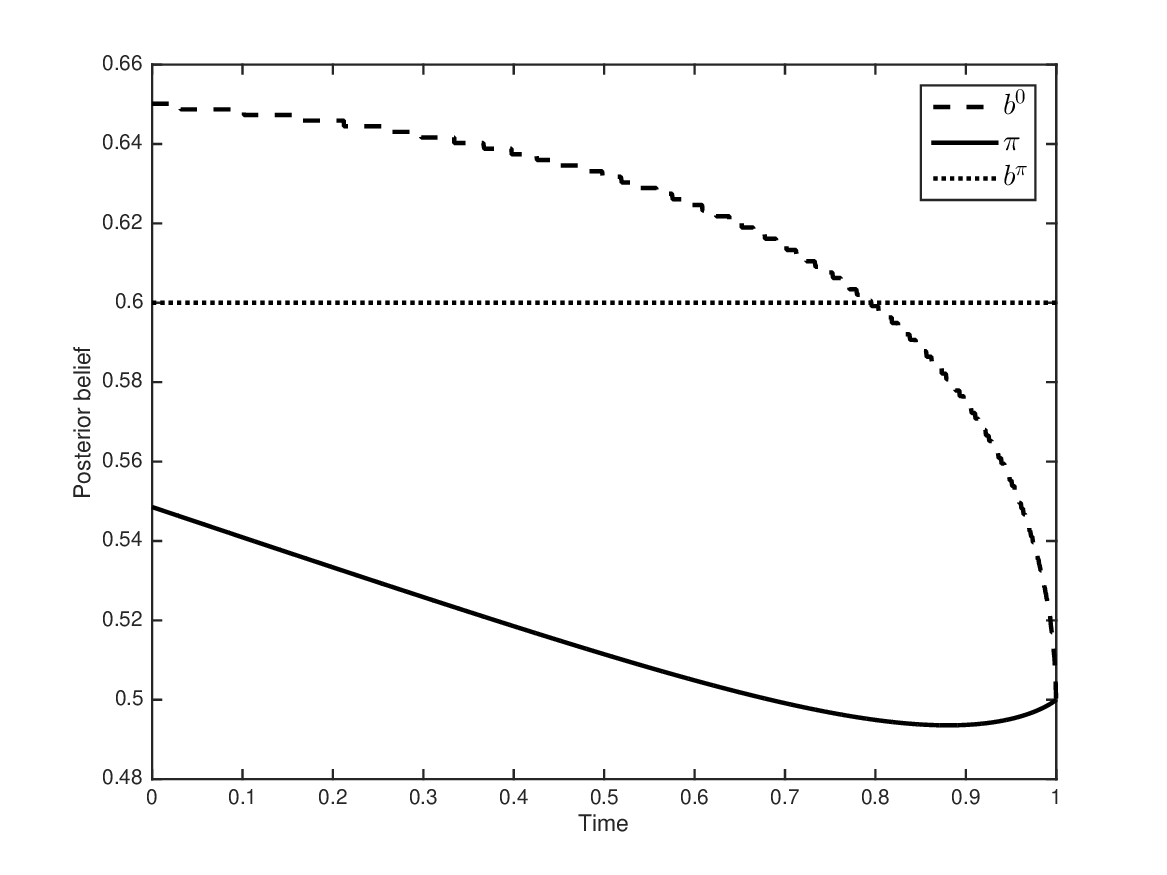}\caption{\label{fig:example}{\scriptsize{}Left: The dashed line depicts the
optimal investment threshold (stopping barrier) $b^{0}$ in the setting
of Example \ref{subsec:Providing-Incentives} when there is no interference
by the government. To incentivize the firm to invest at the first
time when the firm's belief about the investment value $X$ exceeds
the barrier $b^{\pi}\colon[0,T]\to\mathbb{R},b^{\pi}(t)=0.5\sqrt{T-t}$,
the government can use the transfer $\pi\colon[0,T]\to\mathbb{R}$
given by the solid line. The figure is generated using the parameters
$r=1$, $\sigma_{0}^{2}=4$ and $T=1$. Right: The dashed line shows
the optimal stopping barrier $b^{0}$ of the stopping problem (\ref{eq:quick_detect_stopping_prob})
of Example \ref{subsec:quickest_detect}. The solid line depicts a
transfer $\pi\colon[0,T]\to\mathbb{R}$ that induces the agent to
stop at the first time when the posterior belief $p$ exceeds the
level $60\%$. For the figure the parameters $c=\mu=1$ and $T=0.5$
are used. }}
\end{figure}

\subsection{Designing American Options}

Our third examples considers the design of American options. An American
option gives the holder the right, but not the obligation to receive
stocks of a company at a prespecified price over a given time horizon.
The optimal time at which to exercise such an option depends crucially
on the expectations of future stock prices. American options are often
used as managerial compensation. The timing at which a manager exercises
the options given to him, provides information about the managers
expectation of future profitability of the firm. Because of this informational
value, shareholders might want to design American options in such
a way that they provide the manager with incentives to exercise the
option at a given time. Our results imply which exercise timings of
the manager can be incentivized through stock options where the exercise
price is time dependent. 

\section{Problem Formulation\label{sec:Problem-Formulation}}

\subsection{Dynamics}

In this paper we consider optimal stopping problems with finite time
horizon $T<\infty$. The underlying probability space $(\Omega,\mathcal{F},\mathbb{P})$
supports a one-dimensional Brownian motion $W$. Let $\mathbb{F}=\left(\mathcal{F}_{t}\right)_{t\in[0,T]}$
be the filtration generated by $W$ satisfying the usual assumptions.
We denote the set of $\mathbb{F}$-stopping times with values in $[0,T]$
by $\mathcal{T}$. For $t<T$ we refer to $\mathcal{T}_{t,T}$ as
the subset of stopping times which take values in $[t,T].$ The process
$X$ follows the time-inhomogeneous diffusion dynamics 
\begin{equation}
\mathrm{d}X_{t}=\mu(t,X_{t})\mathrm{d}t+\sigma(t,X_{t})\mathrm{d}W_{t}.\label{eq:sde}
\end{equation}
We denote by $\mathcal{L}=\mu\partial_{x}+\frac{1}{2}\sigma^{2}\partial_{xx}$
the infinitesimal generator of $X$. The coefficients $\mu,\sigma:[0,T]\times\mathbb{R}\to\mathbb{R}$
are Borel measurable functions satisfying the following global Lipschitz
and linear growth assumptions: There exists a positive constant $L$
such that
\begin{eqnarray*}
|\mu(t,x)-\mu(t,y)|+|\sigma(t,x)-\sigma(t,y)| & \le & L|x-y|\\
\mu(t,x)^{2}+\sigma(t,x)^{2} & \le & L^{2}(1+x^{2})
\end{eqnarray*}
for all $t\in[0,T]$ and $x,y\in\mathbb{R}$. Under this assumption
there exists a unique strong solution $(X_{s}^{t,x})_{s\ge t}$ to
(\ref{eq:sde}) for every initial condition $X_{t}^{t,x}=x$ (see
e.g. \citet[Theorems 2.5 and 2.9]{karatzas1991brownian}) . Moreover,
it follows that the comparison principle holds true (see e.g. \citet[Proposition 2.18]{karatzas1991brownian}):
The path of the process starting at a lower level $x\le x'$ at time
$t$ is smaller than the path of the process starting in $x'$ at
all later times $s>t$
\begin{equation}
X_{s}^{t,x}\le X_{s}^{t,x'}\qquad\mathbb{P}-a.s.\label{eq:ComparisonPrincipleOriginal}
\end{equation}

\subsection{Payoffs and Transfers}

As long as the process $X$ is not stopped there is a flow payoff
$f$ and at the time of stopping there is a terminal payoff $g$.
The payoffs $f,g:[0,T]\times\mathbb{R}\to\mathbb{R}$ depend on time
and the value of the signal. Formally, the expected payoff for using
a stopping time $\tau\in\mathcal{T}_{t,T}$ equals
\[
W(t,x,\tau)=\mathbb{E}\left[\int_{t}^{\tau}f(s,X_{s}^{t,x})\mathrm{d}s+g(\tau,X_{\tau}^{t,x})\right],
\]
given that $X$ starts in $x\in\mathbb{R}$ at time $t\in[0,T]$.
We assume that the payoff function $f$ is continuous and Lipschitz
continuous in the $x$ variable uniformly in $t$. Moreover, we suppose
that $g\in C^{1,2}([0,T]\times\mathbb{R})$ and that the functions
$g$ and $(\partial_{t}+\mathcal{L})g$ are Lipschitz continuous in
the $x$ variable uniformly in $t$.

We will analyze how preferences over stopping times change if there
is an additional payoff which only depends on time. 
\begin{defn}
A measurable, bounded function $\pi:[0,T]\to\mathbb{R}$ is called
a transfer. 
\end{defn}
We define the value function $v^{\pi}\colon[0,T]\times\mathbb{R}\to\mathbb{R}$
of the stopping problem with payoffs $f$ and $g$ and an additional
transfer $\pi$ by 
\begin{equation}
v^{\pi}(t,x)=\sup_{\tau\in\mathcal{T}_{t,T}}\left(W(t,x,\tau)+\mathbb{E}\left[\pi(\tau)\right]\right).\label{eq:stopping_problem_with_transfer}
\end{equation}
Moreover we introduce for every $t\in[0,T]$ the stopping region 
\[
D_{t}^{\pi}=\left\{ x\in\mathbb{R}\,|\,v^{\pi}(t,x)=g(t,x)+\pi(t)\right\} .
\]

\subsection{Implementability}

A measurable set $A\subset[0,T]\times\mathbb{R}$ is called \textit{time-closed}
if for each time $t\in[0,T]$ the slice $A_{t}=\{x\in\mathbb{R}\,|\,(t,x)\in A\}$
is a closed subset of $\mathbb{R}$ . Let $X$ start in $x\in\mathbb{R}$
at time $t\in[0,T].$ For a time-closed set $A$ we introduce the
first time when $X$ hits $A$ by
\[
\tau_{A}^{t,x}=\inf\left\{ s\ge t\,|\,X_{s}^{t,x}\in A_{s}\right\} \wedge T.
\]
We now come to the definition of implementability.
\begin{defn}[Implementability]
A time-closed set $A$ is implemented by a transfer $\pi$ if the
stopping time $\tau_{A}^{t,x}$ is optimal in (\ref{eq:stopping_problem_with_transfer}),
i.e. for every $t\in[0,T]$ and $x\in\mathbb{R}$ 
\[
v^{\pi}(t,x)=W(t,x,\tau_{A}^{t,x})+\mathbb{E}\left[\pi(\tau_{A}^{t,x})\right]\,.
\]
\end{defn}
For a time-closed set $A$ a necessary condition for implementability
is that each slice $A_{t}$ is included in the stopping region $D_{t}^{\pi}$.
Indeed, let $A$ be implemented by $\pi$ and let $t\in[0,T]$ and
$x\in A_{t}$. Then we have $\tau_{A}^{t,x}=t$. Since $\tau_{A}^{t,x}$
is optimal, this implies $v^{\pi}(t,x)=g(t,x)+\pi(t)$ and hence $x\in D_{t}^{\pi}$.
Consequently, we have $A_{t}\subseteq D_{t}^{\pi}$. 

Observe that the converse inclusion $D_{t}^{\pi}\subseteq A_{t}$
does not necessarily hold true, since optimal stopping times are in
general not unique. At some point $(t,x)\in[0,T]\times\mathbb{R}$
it might be optimal to stop immediately ($x\in D_{t}^{\pi})$ as well
as to wait a positive amount of time until $X$ hits $A$ ($x\notin A_{t}$).
A particularly simple example is the case where $X$ is a martingale
and $f(t,x)=0$ and $g(t,x)=x$. The optional stopping theorem implies
that all stopping times $\tau\in\mathcal{T}_{t,T}$ generate the same
expected payoff $W(t,x,\tau)=x$. Therefore, every set $A$ is implemented
by the zero transfer. The stopping region consists of the whole state
space $D_{t}^{0}=\mathbb{R}$. 

We introduce the notion of strict implementability, where ambiguity
in optimal strategies is ruled out: whenever it is optimal to continue
a positive amount of time it is not optimal to stop.
\begin{defn}[Strict Implementability]
A time-closed set $A$ is strictly implemented by a transfer $\pi$
if $A$ is implemented by $\pi$ and $v^{\pi}(t,x)>g(t,x)+\pi(t)$
for all $x\notin A_{t}$ and $t\in[0,T]$.
\end{defn}
In particular, every strictly implementable set $A$ satisfies $A_{t}=D_{t}^{\pi}$
for the transfer $\pi$. Since the stopping regions $D_{t}^{\pi}$
are closed (see Lemma \ref{lem:continuity} below) the restriction
to time-closed sets is no loss of generality. Any set which is not
time-closed can not be strictly implemented.

Note that the notion of implementability generalizes the notion of
optimal stopping times. If $\tau_{A}^{t,x}$ is an optimal stopping
time in a stopping problem of the form 
\[
\sup_{\tau\in\mathcal{T}_{t,T}}\mathbb{E}\left[\int_{t}^{\tau}f(s,X_{s}^{t,x})\mathrm{d}s+g(\tau,X_{\tau}^{t,x})\right]
\]
 for all $(t,x)\in[0,T]\times\mathbb{R}$, then it is implemented
by the zero transfer. 

\subsection{Single Crossing And Cut-Off Regions}

Next we introduce the main structural condition on the payoff functions.
\begin{condition}[Single-Crossing]
\label{def:SingleCrossing}We say that the single crossing condition
is satisfied if for all $t\in[0,T]$ the mapping $x\mapsto f(t,x)+(\partial_{t}+\mathcal{L})g(t,x)$
is non-increasing. If this monotonicity is strict, then we say that
the strict single crossing condition holds.

Note that the (strict) single crossing condition is satisfied in a
number of examples. For instance it is satisfied in the examples of
Subsections \ref{subsec:Providing-Incentives} and \ref{subsec:quickest_detect}.
\end{condition}
Moreover, we define a special subclass of time-closed sets.\footnote{All our results hold analogously for a lower stopping boundary $A_{t}=(-\infty,b(t)]$
if we impose instead of our single crossing condition that $x\mapsto f(t,x)+(\partial_{t}+\mathcal{L})g(t,x)$
is non-decreasing.}
\begin{defn}
\label{def-cut-off}A time-closed set $A$ is called a cut-off region
if there exists a function $b:[0,T]\to\overline{\mathbb{R}}$ such
that $A_{t}=[b(t),\infty)$. In this case we call $b$ the associated
cut-off and we write 
\[
\tau_{A}^{t,x}=\tau_{b}^{t,x}=\inf\{s\ge t\,|\,X_{s}^{t,x}\ge b(s)\}\wedge T
\]
for $(t,x)\in[0,T]\times\mathbb{R}$. We call $\tau_{b}$ a cut-off
rule. We say that a cut-off region $A$ is regular, if the associated
cut-off $b:[0,T]\to\mathbb{R}$ is càdlàg (i.e. is right continuous
and has left limits in $\mathbb{R}$) and has summable downward jumps,
i.e.
\[
\sum_{0\le s\le t}(\Delta b_{s})^{-}<\infty.
\]
\end{defn}

\section{strictly Implementable Regions are Cut-Off Regions\label{sec:Implemtable_implies_cut-off}}

For optimal stopping problems it is well-known that under the single
crossing condition (or a weaker version of it) there exists a cut-off
rule that is optimal (see e.g. \citet{kotlow1973}, \citet{jacka1992finite}
or \citet{villeneuve2007threshold}). In this section we show that
the opposite direction holds more generally for strict implementability:
Only cut-off regions can be strictly implemented. 

We first state the following regularity result about $v^{\pi}$.
\begin{lem}
\label{lem:continuity}For every transfer $\pi$ and every $t\in[0,T]$
the mapping $x\mapsto v^{\pi}(t,x)$ is Lipschitz continuous. In particular,
the stopping region $D_{t}^{\pi}$ is closed. 
\end{lem}
\begin{proof}
Fix $t\in[0,T]$ and $x,y\in\mathbb{R}.$ By Lipschitz continuity
of $f$ and $g$ there exists a constant $C>0$ such that 
\begin{eqnarray*}
|v^{\pi}(t,x)-v^{\pi}(t,y)| & \le & \sup_{\tau\in\mathcal{T}_{t,T}}\mathbb{E}\left[\int_{t}^{\tau}\left|f(s,X_{s}^{t,x})-f(s,X_{s}^{t,y})\right|ds+\left|g(\tau,X_{\tau}^{t,x})-g(\tau,X_{\tau}^{t,y})\right|\right]\\
 & \le & C\mathbb{E}\left[\sup_{s\in[t,T]}\left|X_{s}^{t,x}-X_{s}^{t,y}\right|\right].
\end{eqnarray*}
By the well-known moment estimate for solutions of stochastic differential
equations (see e.g. \citet[Theorem 3.2]{kunita2004stochastic}) there
exists a constant $\tilde{C}$ such that $\mathbb{E}\left[\sup_{s\in[t,T]}\left|X_{s}^{t,x}-X_{s}^{t,y}\right|\right]\le\tilde{C}|x-y|$.
This yields the claim.
\end{proof}
The next result shows that under the single-crossing condition only
cut-off regions are strictly implementable.
\begin{prop}
\label{prop:only_cut-offs_implementable}Assume that the single crossing
condition holds true. For every transfer $\pi$,
\begin{enumerate}
\item the stopping region $D_{t}^{\pi}$ is a cut-off region
\item and thus if $A$ is strictly implemented by $\pi$ then $A$ is a
cut-off region.
\end{enumerate}
\end{prop}
\begin{proof}
Fix $t\in[0,T].$ First observe that the single crossing condition
implies that $x\mapsto v^{\pi}(t,x)-g(t,x)$ is non-increasing. Indeed,
Itô's formula applied to $g(\cdot,X)$ yields 
\[
W(t,x,\tau)=\mathbb{E}\left[\int_{t}^{\tau}\left(f(s,X_{s}^{t,x})+(\partial_{t}+\mathcal{L})g(s,X_{s}^{t,x})\right)ds+g(t,x)+\int_{t}^{\tau}g_{x}(s,X_{s}^{t,x})\sigma(s,X_{s}^{t,x})dW_{s}\right]
\]
for every $x\in\mathbb{R}$ and $\tau\in\mathcal{T}_{t,T}$. Since
$g_{x}$ is bounded and $\sigma$ has linear growth the process $\int_{t}^{\cdot}g_{x}(s,X_{s}^{t,x})\sigma(s,X_{s}^{t,x})dW_{s}$
is a martingale. It follows from the comparison principle (\ref{eq:ComparisonPrincipleOriginal})
and the single crossing condition that for $x\le y$ 
\begin{eqnarray*}
v^{\pi}(t,x)-g(t,x) & = & \sup_{\tau\in\mathcal{T}_{t,T}}\mathbb{E}\left[\int_{t}^{\tau}\left(f(s,X_{s}^{t,x})+(\partial_{t}+\mathcal{L})g(s,X_{s}^{t,x})\right)\mathrm{d}s+\pi(\tau)\right]\\
 & \ge & \sup_{\tau\in\mathcal{T}_{t,T}}\mathbb{E}\left[\int_{t}^{\tau}\left(f(s,X_{s}^{t,y})+(\partial_{t}+\mathcal{L})g(s,X_{s}^{t,y})\right)\mathrm{d}s+\pi(\tau)\right]\\
 & = & v^{\pi}(t,y)-g(t,y).
\end{eqnarray*}
This implies that $y\in D_{t}^{\pi}$ if $y\ge x$ and $x\in D_{t}^{\pi}$.
Hence $D_{t}^{\pi}$ is an interval which is unbounded on the right.
By Lemma \ref{lem:continuity} the set $D_{t}^{\pi}$ is closed. Hence
there exists some $b(t)\in\overline{\mathbb{R}}$ such that $D_{t}^{\pi}=[b(t),\infty)$.
This implies that $A$ is a cut-off region since $A_{t}=D_{t}^{\pi}$
by the definition of strict implementability.
\end{proof}
We note that if we do not restrict attention to transfers $\pi$ which
depend only on time, but allow for the transfer to depend on the value
of the process $X,$ then any measurable set $A$ can be implemented.
To see this observe that when $\pi(t,x)=-g(x,t)+\mathbf{1}_{\{(x,t)\in A\}}$
the optimal stopping problem becomes
\[
\sup_{\tau\in\mathcal{T}}\mathbb{E}\left[g(X_{\tau},\tau)+\pi(X_{\tau},\tau)\right]=\sup_{\tau\in\mathcal{T}}\mathbb{E}\left[\mathbf{1}_{\{(X_{\tau},\tau)\in A\}}\right]
\]
to which $\tau_{A}=\inf\{t\colon(X_{t},t)\in A\}$ is a solution.
By not allowing for spatial dependence in the transfer the inverse
problem becomes harder to solve. While the assumption of spatial independence
makes our problem mathematically non-trivial it also has a clear economic
motivation in dynamic principal agent applications in economics where
the value of the process is privately observed by the agent and thus
the transfer chosen by the principal can not condition on it.

\section{Implementability of Cut-Off Regions\label{sec:Implementability-of-Cut-Off}}

In this section we prove that the converse implication of Proposition
\ref{prop:only_cut-offs_implementable} holds true as well: Every
regular cut-off region is implementable. We derive a closed form representation
for the transfer in terms of the reflected version of $X$ in Subsection
\ref{subsec:Constrained-Processes}. In Subsection \ref{subsec:Cut-offs_implementable}
we verify that this candidate solution to the inverse optimal stopping
problem indeed implements cut-off regions. The main properties of
the transfer are presented in Subsection \ref{subsec:Properties-of-the-transfer}.
In Subsection \ref{subsec:Uniqueness-of-The-transfer} we provide
a uniqueness result for transfers implementing a cut-off region.

\subsection{Reflected SDEs and a formal derivation of the candidate transfer\label{subsec:Constrained-Processes}}

A solution to a reflected stochastic differential equation (RSDE)
is a pair of processes $(\tilde{X},l)$, where the process $\tilde{X}$
evolves according to the dynamics of the associated SDE (\ref{eq:sde})
below a given barrier $b$ and is pushed below the barrier by the
process $l$ whenever it tries to exceed $b$. Next we give a formal
definition.
\begin{defn}
\label{def:rsde}Let $b$ be a càdlàg barrier, $t\in[0,T]$ a fixed
point in time and $\tilde{\xi}\le b(t)$ a $\mathcal{F}_{t}-$measurable
square-integrable random variable. A pair $(\tilde{X},l)$ of adapted
processes (with càdlàg trajectories) is called a (strong) solution
to the stochastic differential equation (\ref{eq:sde}) reflected
at $b$ with initial condition $(t,\tilde{\xi})$ if it satisfies
the following properties.

\begin{enumerate}
\item $\tilde{X}$ is constrained to stay below the barrier, i.e. $\tilde{X}_{s}\le b(s)$
almost surely for every $s\in[t,T]$.
\item For every $s\in[t,T]$ the following integral equation holds almost
surely 
\begin{equation}
\tilde{X}_{s}=\tilde{\xi}+\int_{t}^{s}\mu(r,\tilde{X}_{r})\mathrm{d}r+\int_{t}^{s}\sigma(r,\tilde{X}_{r})\mathrm{d}W_{r}-l_{s}\,.\label{eq:refldiff}
\end{equation}
\item The process $l$ is non-decreasing and only increases when $\tilde{X}_{t}=b(t),$
i.e.
\begin{equation}
\int_{t}^{T}(b(s)-\tilde{X}_{s})\mathrm{d}l_{s}=0\,.\label{eq:minimality}
\end{equation}
\end{enumerate}
\end{defn}
To stress the dependence of $\tilde{X}$ on the initial value we sometimes
write $\tilde{X}^{t,\tilde{\xi}}$. 
\begin{rem}
Consider the situation where $b$ has a downward jump at time $t$
and $\tilde{X}$ is above $b(t)$ shortly before time $t$, i.e. $\tilde{X}_{t-}(\omega)\in(b(t),b(t-)]$
for some $\omega\in\Omega$. Since $\tilde{X}_{t}\le b(t)$ the reflected
process $\tilde{X}$ has a downward jump at time $t$ as well. Equation
(\ref{eq:refldiff}) implies that $l$ has an upward jump at time
$t.$ Then Equation (\ref{eq:minimality}) yields that $\tilde{X}$
is on the barrier at time $t$, i.e. $\tilde{X}_{t}=b(t)$. Hence,
the jump of $b$ is rather absorbed by $\tilde{X}$ than truly reflected
(which would mean $\tilde{X}_{t}=2b(t)-\tilde{X}_{t-}$). In this
sense $\tilde{X}$ is the maximal version of $X$ which stays below
$b$. This property is crucial in the proof of Theorem \ref{thm:Cut-offs_implementable}.
Existence and uniqueness of $\tilde{X}$ are established in \citet{rutkowski1980stochastic}.
We also refer to \citet{Slominski20101701} who allow for general
modes of reflection. For results about RSDEs with ``true'' jump
reflections we refer to \citet{chaleyat1980reflexion}.
\end{rem}

\subsubsection*{A formal derivation}

Here we establish the link between inverse optimal stopping problems
and RSDEs and derive the representation of a transfer implementing
a cut-off region. To this end assume that the cut-off region $A=[b(t),\infty)$
is implemented by a transfer $\pi$. Without loss of generality we
assume that $\pi(T)=0$ (else take $\tilde{\pi}(t)=\pi(t)-\pi(T)$).
Since we are only interested in a formal derivation here, we make
some regularity assumptions. We assume that the value function of
the stopping problem (\ref{eq:stopping_problem_with_transfer}) is
smooth ($v^{\pi}\in C^{1,2}([0,T]\times\mathbb{R})$) and that $b$
is continuous such that $\tilde{X}$ is continuous as well. Then $v^{\pi}$
satisfies (see e.g. \citep[Chapter IV]{pevskir2006optimal})
\begin{eqnarray*}
\min\left\{ -(\partial_{t}+\mathcal{L})v^{\pi}-f,v^{\pi}-\left(g+\pi\right)\right\}  & = & 0\\
v^{\pi}(T,\cdot) & = & g(T,\cdot)
\end{eqnarray*}
and $b$ is the free boundary of this variational partial differential
equation. In particular, below the cut-off $b$ the value function
$v^{\pi}$ satisfies the continuation equation 
\[
(\partial_{t}+\mathcal{L})v^{\pi}(t,x)=-f(t,x)
\]
for all $x\le b(t).$ On the cut-off, $v^{\pi}$ satisfies the boundary
condition $v^{\pi}(t,b(t))=g(t,b(t))+\pi(t)$ for all $t\in[0,T].$
Moreover, if $b$ is sufficiently regular the smooth fit principle
\[
v_{x}(t,b(t))=g_{x}(t,b(t))
\]
holds for all $t\in[0,T]$ (see e.g. \citet[Section 9.1]{pevskir2006optimal}).
Then Itô's formula implies
\begin{eqnarray*}
\mathbb{E}\left[g(T,\tilde{X}_{T}^{t,b(t)})\right] & = & \mathbb{E}\left[v^{\pi}(T,\tilde{X}_{T}^{t,b(t)})\right]\\
 & = & v^{\pi}(t,b(t))+\mathbb{E}\left[\int_{t}^{T}(\partial_{t}+\mathcal{L})v^{\pi}(s,\tilde{X}_{s}^{t,b(t)})ds-\int_{t}^{T}v_{x}(s,\tilde{X}_{s}^{t,b(t)})dl_{s}\right]\\
 & = & g(t,b(t))+\pi(t)-\mathbb{E}\left[\int_{t}^{T}f(s,\tilde{X}_{s}^{t,b(t)})ds+\int_{t}^{T}g_{x}(s,\tilde{X}_{s}^{t,b(t)})dl_{s}\right].
\end{eqnarray*}
A further application of Itô's formula yields the following representation
of $\pi$
\begin{eqnarray}
\pi(t) & = & \mathbb{E}\left[g(T,\tilde{X}_{T}^{t,b(t)})+\int_{t}^{T}f(s,\tilde{X}_{s}^{t,b(t)})ds+\int_{t}^{T}g_{x}(s,\tilde{X}_{s}^{t,b(t)})dl_{s}\right]-g(t,b(t))\nonumber \\
 & = & \mathbb{E}\left[\int_{t}^{T}f(s,\tilde{X}_{s}^{t,b(t)})+(\partial_{t}+\mathcal{L})g(s,\tilde{X}_{s}^{t,b(t)})\mathrm{d}s\right]\,.\label{eq:candidate_transfer}
\end{eqnarray}
In Theorem \ref{thm:Cut-offs_implementable} below we verify that
Equation (\ref{eq:candidate_transfer}) indeed leads to a transfer
$\pi$ implementing $A$. The proof does neither rely on any analytic
methods nor on results from the theory of partial differential equations.
Instead we employ purely probabilistic arguments based on the single
crossing condition and comparison results for SDEs and RSDEs. This
methodology requires weak regularity assumptions on the model parameters.
In particular there is no ellipticity condition on $\sigma.$ 

\subsubsection*{Properties of RSDEs}

The next proposition proves auxiliary results about RSDEs which we
will use in the proof of Theorem \ref{thm:Cut-offs_implementable}.
There is a broad literature on RSDEs including comparison results
(see e.g. \citet{bo2007strong}). To the best of our knowledge the
comparison principles for RSDE with càdlàg barriers and summable downward
jumps as needed for our result have not been shown before. While all
results follow by standard arguments we give a proof in the Appendix
for the convenience of the reader. For the existence and uniqueness
result we refer to \citet{rutkowski1980stochastic}.
\begin{prop}
\label{prop:refldiff}For every regular\footnote{see Definition \ref{def-cut-off}}
cut-off $b$ there exists a unique strong solution $\tilde{X}$ to
the RSDE (\ref{eq:refldiff}). The process $l$ is given by 
\begin{equation}
l_{s}=\sup_{t\le r\le s}(\tilde{\xi}+\int_{t}^{r}\mu(u,\tilde{X}_{u})\mathrm{d}u+\int_{t}^{r}\sigma(u,\tilde{X}_{u})\mathrm{d}W_{u}-b(r))^{+}.\label{eq:loctime}
\end{equation}
Moreover, $\tilde{X}$ satisfies

\begin{enumerate}
\item \label{enu:momest1}(Square Integrability) $\mathbb{E}\left[\sup_{t\leq s\leq T}(\tilde{X}_{s}^{t,\xi})^{2}\right]<\infty$
for all $t\in[0,T]$.
\item \label{enu:minimality}(Minimality) $\tilde{X}_{s}^{t,\xi}1_{\{s<\tau_{b}\}}=X_{s}^{t,\xi}1_{\{s<\tau_{b}\}}$
a.s. for all $s\in[t,T]$.
\item \label{enu:compprinc}(Comparison Principle for the Reflected Process)If
$\xi_{1}\le\xi_{2}$ a.s., then for $s\in[t,T]$ we have $\tilde{X}_{s}^{t,\xi_{1}}\le\tilde{X}_{s}^{t,\xi_{2}}$
a.s.
\item \label{enu:momest2}(Moment Estimate) For $\xi_{1},\xi_{2}\in L^{2}(\mathcal{F}_{t})$
there exists a constant $K>0$ such that $\mathbb{E}\left[\sup_{t\le r\le s}|\tilde{X}_{r}^{t,\xi_{1}}-\tilde{X}_{r}^{t,\xi_{2}}|^{p}|\mathcal{F}_{t}\right]\le K|\xi_{1}-\xi_{2}|^{p}$
a.s. for all $s\in[t,T]$ and $p=1,2.$
\item \label{enu:comprinc2}(Comparison Principle for the Original Process)
$\tilde{X}_{s}^{t,\xi}\leq X_{s}^{t,\xi}$ a.s. for all $s\in[t,T]$.
\item \label{enu:conv}(Left continuity) Let $t\in[0,T]$ and $x\le b(t)\wedge b(t-)$.
Then $\tilde{X}_{t}^{s,y\wedge b(s)}\to x$ in $L^{2}$ for $s\nearrow t$
and $y\to x$.
\end{enumerate}
\end{prop}
Using similar arguments as in \citet[Chapter V Section 6]{protter2005stochastic}
one can show that $\tilde{X}$ satisfies the strong Markov property.
For $s\ge t$ we define the transition kernel $\tilde{P}_{t,s}$ by
\[
\tilde{P}_{t,s}\varphi(t,x)=\mathbb{E}\left[\varphi(s,\tilde{X}_{s}^{t,x})\right]
\]
for any Borel measurable, bounded function $\varphi:[0,T]\times\mathbb{R}\to\mathbb{R}$.
Then $\tilde{X}$ satisfies for any stopping time $\tau\in\mathcal{T}$
and $u\ge0$
\begin{equation}
\mathbb{E}\left[\varphi(\tau+u,\tilde{X}_{\tau+u})\,|\,\mathcal{F}_{\tau}\right]=\tilde{P}_{\tau,\tau+u}\varphi(\tau,\tilde{X}_{\tau}).\label{eq:strong_markov}
\end{equation}
Moreover, uniqueness of strong solutions of RSDEs implies the following
flow property of $\tilde{X}$. For $t\le r\le s$ and $x\in\mathbb{R}$
we have a.s. 
\begin{equation}
\tilde{X}_{s}^{t,x}=\tilde{X}_{s}^{r,\tilde{X}_{r}^{t,x}}.\label{eq:flow_prop}
\end{equation}

\subsection{Regular Cut-Off Regions are Implementable\label{subsec:Cut-offs_implementable}}

In this section we prove our main theorem stating that every regular
cut-off region is implemented by the transfer derived in Subsection
\ref{subsec:Constrained-Processes}.
\begin{thm}
\label{thm:Cut-offs_implementable}Assume that the single crossing
condition is satisfied. Let $A$ be a regular cut-off region with
boundary $b$. Then it is implemented by the transfer 
\begin{equation}
\pi(t)=\mathbb{E}\left[\int_{t}^{T}f(s,\tilde{X}_{s}^{t,b(t)})\mathrm{+(\partial_{t}+\mathcal{L})g(s,\tilde{X}_{s}^{t,b(t)})d}s\right]\,.\label{eq:DefinitionPayment-2}
\end{equation}
\end{thm}
\begin{proof}
First observe that the cut-off rule $\tau_{b}^{t,x}$ is a stopping
time for all $(t,x)\in[0,T]\times\mathbb{R}$. Indeed, since $X$
has continuous paths and $b$ is right-continuous, the Début-theorem
(see e.g. \citet[Chapter IV, Section 50]{dellacherie1978probabilities})
implies $\tau_{b}^{t,x}\in\mathcal{T}_{t,T}$.

Let $\pi$ be given by Equation (\ref{eq:DefinitionPayment-2}). For
the boundedness and measurability of $\pi$ we refer to Proposition
\ref{prop:Properties_transfer}. We set $h=f+(\partial_{t}+\mathcal{L})g$.
As in the proof of Proposition \ref{prop:only_cut-offs_implementable}
we have 
\[
W(t,x,\tau)=g(t,x)+\mathbb{E}\left[\int_{t}^{\tau}h(s,X_{s}^{t,x})ds\right].
\]
Note that we can write $\pi$ in terms of the transition function
$\tilde{P}$ of $\tilde{X}$ as follows
\[
\pi(t)=\int_{t}^{T}\tilde{P}_{t,s}h(t,b(t))\mathrm{d}s.
\]
The strong Markov property (Equation (\ref{eq:strong_markov})) of
$\tilde{X}$ implies
\[
\tilde{P}_{\tau,\tau+u}h(\tau,b(\tau))=\mathbb{E}\left[h(\tau+u,\tilde{X}_{\tau+u}^{\tau,b(\tau)})\,|\,\mathcal{F}_{\tau}\right]
\]
 for any stopping time $\tau\in\mathcal{T}$ and $u\ge0$. Hence we
have 
\begin{equation}
\pi(\tau)=\mathbb{E}\left[\int_{\tau}^{T}h(s,\tilde{X}_{s}^{\tau,b(\tau)})\mathrm{d}s|\mathcal{F}_{\tau}\right]\,.\label{eq:pitau}
\end{equation}

Fix $t\in[0,T]$ and $x\ge b(t)$. Let $\tau\in\mathcal{T}_{t,T}$
be an arbitrary stopping time. The comparison principle between the
original and the reflected process (Property (\ref{enu:comprinc2}))
implies $X_{s}^{t,x}\ge X_{s}^{t,b(t)}\ge\tilde{X}_{s}^{t,b(t)}$
a.s. for every $s\in[t,T].$ From the flow property (Equation (\ref{eq:flow_prop}))
and the comparison principle for reflected processes (Property (\ref{enu:compprinc}))
follows that $\tilde{X}_{s}^{t,b(t)}=\tilde{X}_{s}^{\tau,\tilde{X}_{\tau}^{t,b(t)}}\le\tilde{X}_{s}^{\tau,b(\tau)}$
a.s. for every $s\in[\tau,T]$. Therefore the single crossing condition
implies
\begin{eqnarray*}
\mathbb{E}\left[\int_{t}^{\tau}h(s,X_{s}^{t,x})\mathrm{d}s+\pi(\tau)\right] & = & \mathbb{E}\left[\int_{t}^{\tau}h(s,X_{s}^{t,x})\mathrm{d}s+\int_{\tau}^{T}h(s,\tilde{X}_{s}^{\tau,b(\tau)})\mathrm{d}s\right]\\
 & \le & \mathbb{E}\left[\int_{t}^{\tau}h(s,\tilde{X}_{s}^{t,b(t)})\mathrm{d}s+\int_{\tau}^{T}h(s,\tilde{X}_{s}^{t,b(t)})\mathrm{d}s\right]\\
 & = & \pi(t).
\end{eqnarray*}
This implies $W(t,x,\tau)+\mathbb{E}\left[\pi(\tau)\right]\le W(t,x,t)+\pi(t)$.
Hence $\tau_{b}^{t,x}=t$ is optimal in (\ref{eq:stopping_problem_with_transfer})
as claimed. 

In the second step fix $x<b(t)$ and let $\tau\in\mathcal{T}_{t,T}$
be an arbitrary stopping time. To shorten notation we write $\tau_{b}=\tau_{b}^{t,x}$.
First, we prove that the stopping $\min\{\tau,\tau_{b}\}$ performs
at least as well as $\tau.$ By (\ref{eq:pitau}) we have 
\begin{eqnarray*}
\mathbb{E}\left[1_{\{\tau_{b}<\tau\}}\pi(\tau)\right] &  & =\mathbb{E}\left[1_{\{\tau_{b}<\tau\}}\mathbb{E}\left[\int_{\tau}^{T}h(s,\tilde{X}_{s}^{\tau,b(\tau)})\mathrm{d}s\,|\,\mathcal{F}_{\tau}\right]\right]=\mathbb{E}\left[1_{\{\tau_{b}<\tau\}}\int_{\tau}^{T}h(s,\tilde{X}_{s}^{\tau,b(\tau)})\mathrm{d}s\right].
\end{eqnarray*}
 This leads to
\begin{flalign*}
\mathbb{E} & \left[1_{\{\tau_{b}<\tau\}}\left(\int_{t}^{\tau}h(s,X_{s}^{t,x})\mathrm{d}s+\pi(\tau)\right)\right]\\
 & =\mathbb{E}\left[1_{\{\tau_{b}<\tau\}}\left(\int_{t}^{\tau_{b}}h(s,X_{s}^{t,x})\mathrm{d}s+\int_{\tau_{b}}^{\tau}h(s,X_{s}^{t,x})\mathrm{d}s+\int_{\tau}^{T}h(s,\tilde{X}_{s}^{\tau,b(\tau)})\mathrm{d}s\right)\right].
\end{flalign*}
By construction of the reflected process $\tilde{X}$ we have $\tilde{X}_{\tau_{b}}^{t,x}=b(\tau_{b})$.
The comparison principle between the original and the reflected process
(Property (\ref{enu:comprinc2})) and the flow property of reflected
processes (Equation (\ref{eq:flow_prop})) imply almost surely 
\[
\tilde{X}_{s}^{\tau_{b},b(\tau_{b})}=\tilde{X}_{s}^{\tau_{b},\tilde{X}_{\tau_{b}}^{t,x}}=\tilde{X}_{s}^{t,x}\le X_{s}^{t,x}
\]
for $s\ge\tau_{b}$. Since $\tilde{X}_{\tau}^{\tau_{b},b(\tau_{b})}\le b(\tau)$
we have on the set $\{\tau>\tau_{b}\}$ 
\[
\tilde{X}_{s}^{\tau,b(\tau)}\ge\tilde{X}_{s}^{\tau,\tilde{X}_{\tau}^{\tau_{b},b(\tau_{b})}}=\tilde{X}_{s}^{\tau_{b},b(\tau_{b})}
\]
for all $s\ge\tau$. These two inequalities combined with the monotonicity
of $h$ yield that 
\begin{flalign*}
\mathbb{E} & \left[1_{\{\tau_{b}<\tau\}}\left(\int_{t}^{\tau}h(s,X_{s}^{t,x})\mathrm{d}s+\pi(\tau)\right)\right]\\
 & \le\mathbb{E}\left[1_{\{\tau_{b}<\tau\}}\left(\int_{t}^{\tau_{b}}h(s,X_{s}^{t,x})\mathrm{d}s+\int_{\tau_{b}}^{\tau}h(s,\tilde{X}_{s}^{\tau_{b},b(\tau_{b})})\mathrm{d}s+\int_{\tau}^{T}h(s,\tilde{X}_{s}^{\tau_{b},b(\tau_{b})})\mathrm{d}s\right)\right]\\
 & =\mathbb{E}\left[1_{\{\tau_{b}<\tau\}}\left(\int_{t}^{\tau_{b}}h(s,X_{s}^{t,x})\mathrm{d}s+\pi(\tau_{b})\right)\right]\,.
\end{flalign*}
Consequently using the stopping time $\min\{\tau,\tau_{b}\}$ is at
least as good as using $\tau$ 
\begin{flalign*}
W(t,x,\tau)+\mathbb{E}\left[\pi(\tau)\right] & =g(t,x)+\mathbb{E}\left[\int_{t}^{\tau}h(s,X_{s}^{t,x})\mathrm{d}s+\pi(\tau)\right]\\
 & \le g(t,x)+\mathbb{E}\left[\int_{t}^{\tau\wedge\tau_{b}}h(s,X_{s}^{t,x})\mathrm{d}s+\pi(\min\{\tau,\tau_{b}\})\right]\\
 & =W(t,x,\min\{\tau,\tau_{b}\})+\mathbb{E}\left[\pi(\min\{\tau,\tau_{b}\})\right].
\end{flalign*}
Thus it suffices to consider stopping rules $\tau\leq\tau_{b}$. In
this case we have
\begin{flalign*}
\mathbb{E} & \left[\int_{t}^{\tau}h(s,X_{s}^{t,x})\mathrm{d}s+\pi(\tau)\right]\\
 & =\mathbb{E}\left[\int_{t}^{\tau}h(s,X_{s}^{t,x})\mathrm{d}s+\int_{\tau}^{\tau_{b}}h(s,\tilde{X}_{s}^{\tau,b(\tau)})\mathrm{d}s+\int_{\tau_{b}}^{T}h(s,\tilde{X}_{s}^{\tau,b(\tau)})\mathrm{d}s\right].
\end{flalign*}
From the comparison principle for reflected processes (Property (\ref{enu:compprinc}))
and the flow property Equation (\ref{eq:flow_prop}) follows $\tilde{X}_{s}^{t,x}=\tilde{X}_{s}^{\tau,\tilde{X}_{\tau}^{t,x}}\le\tilde{X}_{s}^{\tau,b(\tau)}$
for all $s\ge\tau.$ By the minimality property of reflected processes
(Property (\ref{enu:minimality})) we have that $X_{s}^{t,x}=\tilde{X}_{s}^{t,x}$
for all $s<\tau_{b}$. Similar considerations as above yield
\[
\tilde{X}_{s}^{\tau_{b},b(\tau_{b})}=\tilde{X}_{s}^{\tau_{b},\tilde{X}_{\tau_{b}}^{t,x}}=\tilde{X}_{s}^{t,x}=\tilde{X}_{s}^{\tau,\tilde{X}_{\tau}^{t,x}}\le\tilde{X}_{s}^{\tau,b(\tau)}
\]
a.s. for $s\ge\tau_{b}$. The monotonicity of $h$ implies 
\begin{flalign*}
\mathbb{E}\left[\int_{t}^{\tau}h(s,X_{s}^{t,x})\mathrm{d}s+\pi(\tau)\right] & \le\mathbb{E}\left[\int_{t}^{\tau}h(s,X_{s}^{t,x})\mathrm{d}s+\int_{\tau}^{\tau_{b}}h(s,X_{s}^{t,x})\mathrm{d}s+\int_{\tau_{b}}^{T}h(s,\tilde{X}_{s}^{\tau_{b},b(\tau_{b})})\mathrm{d}s\right]\\
 & =\mathbb{E}\left[\int_{t}^{\tau_{b}}h(s,X_{s}^{t,x})\mathrm{d}s+\pi(\tau_{b})\right]
\end{flalign*}
and hence $W(t,x,\tau)+\mathbb{E}\left[\pi(\tau)\right]\le W(t,x,\tau_{b})+\mathbb{E}\left[\pi(\tau_{b})\right]$.
This completes the proof of implementability.
\end{proof}

Theorem \ref{thm:Cut-offs_implementable} shows that every cut-off
stopping time is implementable under the single crossing condition
we imposed. We note that this result does not hold without the single
crossing condition. To see this consider as an example a payoff $g(x,t)=h(|x|,t)$
which is only a function of the absolute value of $x$ and a symmetric
diffusion process $\mu(x,t)=-\mu(-x,t)$ and $\sigma(x,t)=\sigma(-x,t)$.
Note, that such an example never satisfies the single crossing condition.
As for any $\pi:\mathbb{R}_{+}\to\mathbb{R}$ the optimal stopping
problem
\[
\sup_{\tau\in\mathcal{T}}\mathbb{E}\left[g(X_{\tau},\tau)+\pi(\tau)\right]
\]
is symmetric at zero it follows that the stopping set must be symmetric
around zero. Consequently, the agent does not only stop when the process
$X$ crosses a threshold from below, but also when $X$ crosses the
negative of this threshold from above. Hence, the optimal stopping
time is never of cut-off form, and no cut-off rule can be implemented.

In Proposition \ref{prop:only_cut-offs_implementable} we showed that
strictly implementable regions are necessarily of cut-off type. The
next result establishes the converse direction. Under the strict single
crossing condition cut-off regions are strictly implementable. 
\begin{cor}
\label{thm:strict_single_crossing_impl_strong_impl}If the strict
single crossing condition holds true, then a regular cut-off region
with barrier $b$ is strictly implemented by the transfer from Equation
(\ref{eq:DefinitionPayment-2}).
\end{cor}
\begin{proof}
We use the same notation as in the proof of Theorem \ref{thm:Cut-offs_implementable}.
Let $t\in[0,T]$ and $x<b(t)$. Then the right-continuity of $b$
and $\tilde{X}$ and the strict monotonicity of $h$ imply that 
\[
\mathbb{E}\left[\int_{t}^{\tau_{b}}h(s,\tilde{X}_{s}^{t,x})\mathrm{d}s\right]>\mathbb{E}\left[\int_{t}^{\tau_{b}}h(s,\tilde{X}_{s}^{t,b(t)})\mathrm{d}s\right]
\]
 Consequently we have
\begin{eqnarray*}
\mathbb{E}\left[\int_{t}^{\tau_{b}}h(s,X_{s}^{t,x})\mathrm{d}s+\pi(\tau_{b})\right] & = & \mathbb{E}\left[\int_{t}^{\tau_{b}}h(s,\tilde{X}_{s}^{t,x})\mathrm{d}s+\int_{\tau_{b}}^{T}h(s,\tilde{X}_{s}^{\tau_{b},b(\tau_{b})})\mathrm{d}s\right]\\
 & > & \mathbb{E}\left[\int_{t}^{\tau_{b}}h(s,\tilde{X}_{s}^{t,b(t)})\mathrm{d}s+\int_{\tau_{b}}^{T}h(s,\tilde{X}_{s}^{t,b(t)})\mathrm{d}s\right]\\
 & = & \pi(t).
\end{eqnarray*}
 This implies $v^{\pi}(t,x)>\pi(t)+g(t,x)$ and hence $A$ is strictly
implemented by $\pi.$
\end{proof}
In general the distribution of the reflected process $\tilde{X}$
is not explicitly known. Hence, one has to fall back to numerical
methods to approximate the transfer from Theorem \ref{thm:Cut-offs_implementable}.
For example one could use discretization schemes for the RSDE (\ref{eq:refldiff})
and Monte Carlo simulations to evaluate the expectation in Equation
(\ref{eq:DefinitionPayment-2}) (see e.g. \citet{saisho1987stochastic},
\citet{bossy2004symmetrized} or \citet{onskog2010weak}). If $X$
evolves according to a Brownian motion, then the distribution of $\tilde{X}$
is available in closed form.

\subsection{Properties of the Transfer\label{subsec:Properties-of-the-transfer}}

The next proposition summarizes properties of transfer implementing
a cut-off region. 
\begin{prop}
\label{prop:Properties_transfer}Let $b:[0,T]\to\mathbb{R}$ be a
regular cut-off. The transfer $\pi$ from Equation (\ref{eq:DefinitionPayment-2})
satisfies the following properties 

\begin{enumerate}
\item $\pi$ is càdlàg. In particular $\pi$ is bounded and measurable. 
\item $\pi$ is continuous at $t\in[0,T]$ if $b$ is continuous at $t$
or if $b$ has a downward jump at $t$.
\item $\pi$ has no upward jumps.
\item If $\pi$ has a downward jump at $t\in[0,T]$, then $b$ has an upward
jump at $t$.
\item $\pi$ converges to $0$ at time $T$: $\lim_{t\nearrow T}\pi(t)=0$.
\end{enumerate}
\end{prop}
\begin{proof}
As in the proof of Theorem \ref{thm:Cut-offs_implementable} we introduce
the function $h(t,x)=f(t,x)+\left(\partial_{t}+\mathcal{L}\right)g(t,x)$.
By assumption $h$ is Lipschitz continuous and has linear growth in
$x$. The transfer $\pi$ is given by
\[
\pi(t)=\mathbb{E}\left[\int_{t}^{T}h(s,\tilde{X}_{s}^{t,b(t)})ds\right].
\]
We first show that $\pi$ is right-continuous. For $t\in[0,T]$ and
$\epsilon>0$ we have
\[
\left|\pi(t)-\pi(t+\epsilon)\right|\le\mathbb{E}\left[\int_{t}^{t+\epsilon}\left|h(s,\tilde{X}_{s}^{t,b(t)})\right|ds\right]+\mathbb{E}\left[\int_{t+\epsilon}^{T}\left|h(s,\tilde{X}_{s}^{t,b(t)})-h(s,\tilde{X}_{s}^{t+\epsilon,b(t+\epsilon)})\right|ds\right].
\]
It follows from the linear growth of $h$ and Property (\ref{enu:momest1})
of $\tilde{X}$ from Proposition \ref{prop:refldiff} that $\mathbb{E}\left[\int_{t}^{t+\epsilon}\left|h(s,\tilde{X}_{s}^{t,b(t)})\right|ds\right]\to0$
as $\epsilon\to0$. Moreover, the Lipschitz continuity of $h$ implies
\[
\mathbb{E}\left[\int_{t+\epsilon}^{T}\left|h(s,\tilde{X}_{s}^{t,b(t)})-h(s,\tilde{X}_{s}^{t+\epsilon,b(t+\epsilon)})\right|ds\right]\le C\mathbb{E}\left[\sup_{s\in[t+\epsilon,T]}\left|\tilde{X}_{s}^{t,b(t)}-\tilde{X}_{s}^{t+\epsilon,b(t+\epsilon)}\right|\right]
\]
for some constant $C>0$. By the flow property (Equation (\ref{eq:flow_prop}))
we have $\tilde{X}_{s}^{t,b(t)}=\tilde{X}_{s}^{t+\epsilon,\tilde{X}_{t+\epsilon}^{t,b(t)}}$.
Property (\ref{enu:momest2}) from Proposition \ref{prop:refldiff}
yields 
\[
\mathbb{E}\left[\sup_{s\in[t+\epsilon,T]}\left|\tilde{X}_{s}^{t,b(t)}-\tilde{X}_{s}^{t+\epsilon,b(t+\epsilon)}\right|\right]\le\tilde{C}\mathbb{E}\left[\left|\tilde{X}_{t+\epsilon}^{t,b(t)}-b(t+\epsilon)\right|\right].
\]
Right continuity of $\tilde{X}$ and $b$ then implies $\pi(t+)=\pi(t)$.\footnote{Here and in the sequel we use the notation $\pi(t+)=\lim_{\epsilon\searrow0}\pi(t+\epsilon)$
and $\pi(t-)=\lim_{\epsilon\searrow0}\pi(t-\epsilon)$ for the one-sided
limits.}

Concerning the left-hand limits of $\pi$ we show that 
\begin{equation}
\pi(t-)=\mathbb{E}\left[\int_{t}^{T}h(s,\tilde{X}_{s}^{t,b(t)\wedge b(t-)})ds\right].\label{eq:Left-continuity}
\end{equation}
for all $t\in(0,T].$ Equation (\ref{eq:Left-continuity}) implies
all remaining claims of Proposition \ref{prop:Properties_transfer}.
If $b$ is continuous at $t$ or has a downward jump ($b(t)\le b(t-)$),
then Equation (\ref{eq:Left-continuity}) yields continuity of $\pi$
at $t$: $\pi(t-)=\pi(t)$. Monotonicity of $h$ and the comparison
principle for the reflected process imply $\pi(t-)\ge\pi(t),$ i.e.
$\pi$ has no upward jumps. If $\pi$ has a downward jump at time
$t$ ($\pi(t-)>\pi(t)$), then Equation (\ref{eq:Left-continuity})
yields that $b$ has necessarily an upward jump ($b(t)>b(t-)$). Moreover,
it follows from Equation (\ref{eq:Left-continuity}) that $\pi(T-)=0$.
To prove Equation (\ref{eq:Left-continuity}) let $t\in(0,T]$ and
$\epsilon>0$. Then consider
\begin{eqnarray*}
\left|\pi(t-\epsilon)-\mathbb{E}\left[\int_{t}^{T}h(s,\tilde{X}_{s}^{t,b(t)\wedge b(t-)})ds\right]\right| & \le & \mathbb{E}\left[\int_{t-\epsilon}^{t}\left|h(s,\tilde{X}_{s}^{t-\epsilon,b(t-\epsilon)})\right|ds\right]\\
 &  & +\mathbb{E}\left[\int_{t}^{T}\left|h(s,\tilde{X}_{s}^{t-\epsilon,b(t-\epsilon)})-h(s,\tilde{X}_{s}^{t,b(t)\wedge b(t-)})\right|ds\right].
\end{eqnarray*}
By Property (\ref{enu:conv}) from Proposition \ref{prop:refldiff}
we have $\tilde{X}_{s}^{t-\epsilon,b(t-\epsilon)}\to\tilde{X}_{s}^{t,b(t)\wedge b(t-)}$
in $L^{2}$ as $\epsilon\searrow0$. Lipschitz continuity and linear
growth of $h$ then imply that $\mathbb{E}\left[\int_{t-\epsilon}^{t}\left|h(s,\tilde{X}_{s}^{t-\epsilon,b(t-\epsilon)})\right|ds\right]\to0$
and $\mathbb{E}\left[\int_{t}^{T}\left|h(s,\tilde{X}_{s}^{t-\epsilon,b(t-\epsilon)})-h(s,\tilde{X}_{s}^{t,b(t)\wedge b(t-)})\right|ds\right]\to0$
for $\epsilon\searrow0$. This yields the claim.
\end{proof}

\subsection{Uniqueness of the Transfer \label{subsec:Uniqueness-of-The-transfer}}

To prove a uniqueness result for the transfer from Theorem \ref{thm:Cut-offs_implementable}
we need the following auxiliary result about cut-off stopping times.
\begin{lem}
\label{lem:convergence_stopp_times}Let $b:[0,T]\to\mathbb{R}$ be
bounded from below. Then we have $\tau_{b}^{t,x}\nearrow T$ a.s.
for $x\searrow-\infty$ and for every $t\in[0,T].$
\end{lem}
\begin{proof}
Fix $t\in[0,T]$. By \citet[Lemma 3.7]{kunita2004stochastic} there
exists a constant $C>0$ such that
\[
\mathbb{E}\left[\sup_{t\le s\le T}\left(\frac{1}{1+(X_{s}^{t,x})^{2}}\right)^{2}\right]\le C\left(\frac{1}{1+x^{2}}\right)^{2}.
\]
Then Fatou's Lemma implies
\[
\mathbb{E}\left[\liminf_{x\to-\infty}\sup_{t\le s\le T}\left(\frac{1}{1+(X_{s}^{t,x})^{2}}\right)^{2}\right]\le\liminf_{x\to-\infty}\mathbb{E}\left[\sup_{t\le s\le T}\left(\frac{1}{1+(X_{s}^{t,x})^{2}}\right)^{2}\right]\le\liminf_{x\to-\infty}C\left(\frac{1}{1+x^{2}}\right)^{2}=0.
\]
Consequently we have $\limsup_{x\to-\infty}\inf_{t\le s\le T}|X_{s}^{t,x}|=\infty$
a.s.  Together with the comparison principle for $X$ this yields
$\limsup_{x\to-\infty}\sup_{t\le s\le T}X_{s}^{t,x}=-\infty$ a.s.
It follows that $\tau_{b}^{t,x}\nearrow T$ for $x\searrow-\infty.$
\end{proof}
\begin{thm}
\label{thm:Uniqueness}Let $A$ be a regular cut-off region with boundary
$b$. Assume that $A$ is implemented by two transfers $\pi$ and
$\hat{\pi}$ satisfying $\lim_{t\nearrow T}\pi(t)=\lim_{t\nearrow T}\hat{\pi}(t)$.
Then $\pi(t)=\hat{\pi}(t)$ for all $t\in[0,T)$.
\end{thm}
\begin{proof}
Fix $t\in[0,T).$ To shorten notation we set $v=v^{\pi}$ and $\hat{v}=v^{\hat{\pi}}.$
By Lemma \ref{lem:continuity} the functions $v$ and $\hat{v}$ are
Lipschitz continuous in the $x$ variable. Similar considerations
yield that the function $x\mapsto W(t,x,\tau)$ is Lipschitz continuous
for every $\tau\in\mathcal{T}_{t,T}$. In particular, these functions
are absolutely continuous. Appealing to the envelope theorem from
\citet[Theorem 1]{milgrom2002envelope} yields that 
\[
v_{x}(t,x)=W_{x}(t,x,\tau_{b}^{t,x})=\hat{v}_{x}(t,x)
\]
for Lebesgue almost every $x\in\mathbb{R}$. Integrating from $x<b(t)$
to $b(t$) gives
\[
v(t,b(t))-v(t,x)=\hat{v}(t,b(t))-\hat{v}(t,x)
\]
or equivalently

\[
\pi(t)-\hat{\pi}(t)=\mathbb{E}\left[\pi(\tau_{b}^{t,x})-\hat{\pi}(\tau_{b}^{t,x})\right].
\]
Since $\pi$ and $\hat{\pi}$ are bounded we can appeal to Lemma \ref{lem:convergence_stopp_times}
to obtain
\[
\pi(t)-\hat{\pi}(t)=\lim_{x\to-\infty}\mathbb{E}\left[\pi(\tau_{b}^{t,x})-\hat{\pi}(\tau_{b}^{t,x})\right]=0,
\]
 where we used the dominated convergence theorem.
\end{proof}

\section{Application To Optimal Stopping\label{sec:Application-To-Optimal-Stopping}}

From Theorem \ref{thm:Uniqueness} we derive a probabilistic characterization
of optimal stopping times for stopping problems of the form
\begin{equation}
v(t,x)=\sup_{\tau\in\mathcal{T}_{t,T}}\mathbb{E}\left[\int_{t}^{\tau}f(s,X_{s}^{t,x})ds+g(\tau,X_{\tau}^{t,x})\right]\,,\label{eq:optstop}
\end{equation}
where $f,g$ and $X$ satisfy the single crossing condition. We say
that a stopping time $\tau\in\mathcal{T}_{t,T}$ is optimal in (\ref{eq:optstop})
for $(t,x)\in[0,T]\times\mathbb{R}$ if 
\[
v(t,x)=\mathbb{E}\left[\int_{t}^{\tau}f(s,X_{s}^{t,x})ds+g(\tau,X_{\tau}^{t,x})\right].
\]

\begin{cor}
\label{cor:appl_opt_stopping}Assume that the single crossing condition
is satisfied and let $b:[0,T]\to\mathbb{R}$ be a regular cut-off.
The stopping time $\tau_{b}^{t,x}$ is optimal in (\ref{eq:optstop})
for all $(t,x)\in[0,T]\times\mathbb{R}$, if and only if $b$ satisfies
the nonlinear integral equation 
\begin{equation}
\mathbb{E}\left[\int_{t}^{T}f(s,\tilde{X}_{s}^{t,b(t)})+(\partial_{t}+\mathcal{L})g(s,\tilde{X}_{s}^{t,b(t)})\mathrm{d}s\right]=0\label{eq:integral_equation_reflected}
\end{equation}
for all $t\in[0,T].$
\end{cor}
\begin{proof}
First assume that (\ref{eq:integral_equation_reflected}) holds true
for every $t\in[0,T]$. Then Theorem \ref{thm:Cut-offs_implementable}
implies that the cut-off region with boundary $b$ is implemented
by the zero transfer. This means that $\tau_{b}^{t,x}$ is optimal
in (\ref{eq:optstop}) for every $(t,x)\in[0,T]\times\mathbb{R}$.

For the converse direction assume that $\tau_{b}^{t,x}$ is optimal
in (\ref{eq:optstop}) for every $(t,x)\in[0,T]\times\mathbb{R}$.
Then the cut-off region with boundary $b$ is implemented by the zero
transfer $\hat{\pi}=0$. By Theorem \ref{thm:Cut-offs_implementable}
it is also implemented by the transfer $\pi(t)=\mathbb{E}\left[\int_{t}^{T}f(s,\tilde{X}_{s}^{t,b(t)})+(\partial_{t}+\mathcal{L})g(s,\tilde{X}_{s}^{t,b(t)})\mathrm{d}s\right]$.
By Proposition \ref{prop:Properties_transfer} the transfer $\pi$
satisfies $\lim_{t\nearrow T}\pi(t)=0$. Then Theorem \ref{thm:Uniqueness}
implies that $\pi(t)=\hat{\pi}(t)=0$ for all $t\in[0,T].$
\end{proof}
In the literature on optimal stopping there is a well known link between
optimal stopping boundaries and a nonlinear integral equation differing
from Equation (\ref{eq:integral_equation_reflected}). It was independently
derived by \citet{kim1990analytic}, \citet{jacka1991optimal} and
\citet{carr2006alternative} who considered the optimal exercise of
an American option. Using the \textit{early exercise premium representation}
of the price of an American option, the authors arrive at a nonlinear
integral equation that is satisfied by the optimal exercise boundary.
The question whether the optimal exercise boundary is the only solution
to the integral equation was left open, until more than a decade later
\citet{peskir2005american} answered it in the affirmative. Using
the change-of-variable formula with local time on curves derived in
\citet{peskir2005change}, allows \citet{peskir2005american} to characterize
the optimal exercise boundary as the unique solution of the nonlinear
integral equation in the class of continuous functions. The methodology
of \citet{peskir2005american} was subsequently applied to solve optimal
stopping problems with more general diffusion and Markov processes,
multiple stopping boundaries and more general payoff functionals.
These problems include for example the optimal exercise of Russian
(\citet{peskir2005russian}) and British options (\citet{peskir2011british,peskir2013british}),
the Wiener disorder problem (\citet{gapeev2006wiener}), sequential
testing problems (\citet{gapeev2004wiener}, \citet{zhitlukhin2013chernoff}),
the optimal stopping problem for maxima in diffusion models (\citet{gapeev2006discounted}),
optimal prediction problems (\citet{du2007trap}), Bayesian disorder
problems (\citet{zhitlukhin2013bayesian}), optimal liquidation problems
(\citet{ekstrom2016optimal}) and multiple optimal stopping problems
(\citet{de2014optimal}). In the framework of the present paper this
integral equation is given by 
\begin{equation}
\mathbb{E}\left[\int_{t}^{T}\left(f(s,X_{s}^{t,b(t)})+(\partial_{t}+\mathcal{L})g(s,X_{s}^{t,b(t)})\right)\mathbf{1}_{\{X_{s}^{t,b(t)}\le b(s)\}}\mathrm{d}s\right]=0\label{eq:integraleq}
\end{equation}
(cf. \citet[Chapter IV, Section 14]{pevskir2006optimal}). 

Besides its interpretation in terms of the early exercise premium,
Equation (\ref{eq:integraleq}) is valuable from a numerical point
of view. Indeed, it only requires for every $s\in[0,T]$ the law of
$X_{s}$ which, when not known explicitly, can be approximated in
various ways (e.g. using the Kolmogorov forward equation or Euler-Maruyama
schemes). Once these distributions are available, (\ref{eq:integraleq})
is a nonlinear Volterra (or Fredholm) integral equation which can
be tackled using well-established numerical schemes provided in the
literature. We also refer to the work of \citet{belomestny2010iterative},
where an iterative procedure is proposed to approximate the solution
of the integral equation and the value function. In contrast, it is
not clear whether (\ref{eq:integral_equation_reflected}) can be numerically
solved with high accuracy, since it is path-dependent in terms of
$b$. In particular, it is not possible to compute the marginal laws
of $\tilde{X}$ upfront, as the unknown boundary $b$ is entangled
in the process $\tilde{X}$ (the distribution of the random variable
$\tilde{X}_{s}^{t,b(t)}$ depends on the whole barrier $(b(r))_{t\le r\le s}$
from time $t$ to $s$). 

We also mention that the change of variables formula of \citet{peskir2005change},
was extended in \citet{peskir2007change} to the multidimensional
setting. This allows to characterize optimal stopping times as first
hitting times also in higher dimensions (see e.g. \citet{glover2013three},
\citet{gapeev2013} and \citet{peskir2014quickest}). Whether an extension
of the methodology presented here to the multidimensional setting
is possible is not clear. Already the formulation of the monotonicity
in the single crossing condition (Condition \ref{def:SingleCrossing})
in higher dimensions is not straightforward. The construction of multivariate
reflected processes is also highly nontrivial.

In general, the set of solutions to (\ref{eq:integral_equation_reflected})
is included in the set of solutions to (\ref{eq:integraleq}). Indeed,
if $b$ solves (\ref{eq:integral_equation_reflected}) then by Corollary
\ref{cor:appl_opt_stopping} it is an optimal stopping boundary and
thus, under appropriate regularity conditions, it is also a solution
to (\ref{eq:integraleq}). In the cases where uniqueness holds for
(\ref{eq:integraleq}) (see the list of references above), the converse
implication holds true as well. In the case of a constant barrier
$b(t)=b\in\mathbb{R}$ and $X$ a Brownian motion one can directly
relate the two equations. Indeed, in this case it follows from the
reflection principle that for all $x\leq b$
\[
\mathbb{P}\left[\tilde{X}_{s}^{t,b(t)}\leq x\right]=2\,\mathbb{P}\left[X_{s}^{t,b(t)}\leq x\right]
\]
 and thus the constant barrier $b$ solves Equation (\ref{eq:integraleq})
if and only if it solves Equation (\ref{eq:integral_equation_reflected}).
The question whether one can in general relate the two integral equations
without taking the detour via optimal stopping problems is left open
for future research.

\section{Appendix}
\begin{proof}[Proof of Proposition \ref{prop:refldiff}]
Existence and uniqueness of $(\tilde{X},l)$ follow from \citet{rutkowski1980stochastic}.
See also \citet[Theorem 3.4]{Slominski20101701} for the time-homogeneous
case. By construction of $(\tilde{X},l)$ we also have (\ref{enu:momest1}).

We next show (\ref{enu:minimality}). Note that the solution to the
unreflected SDE (\ref{eq:sde}) solves the reflected SDE for $s<\tau_{b}$.
As the solution to the reflected SDE is unique (\ref{enu:minimality})
follows.

To prove (\ref{enu:compprinc}) and (\ref{enu:momest2}) we consider
without loss of generality only the case $t=0$. For $\xi_{1},\xi_{2}\in\mathbb{R}$
we write $(\tilde{X}^{i},l^{i})=(\tilde{X}^{0,\xi_{i}},l^{0,\xi_{i}})$,
($i=1,2$) and introduce the processes $D_{t}=\tilde{X}_{t}^{1}-\tilde{X}_{t}^{2}$
and $\Gamma_{t}=\sup_{s\le t}\max(0,D_{s})^{2}$. Applying the Meyer-Itô
formula \citet[Theorem 71, Chapter 4]{protter2005stochastic} to the
function $x\mapsto\max(0,x)^{2}$ yields
\begin{equation}
\begin{split}\max(0,D_{s})^{2}= & \max(0,D_{0})^{2}+2\int_{0}^{s}1_{\{D_{r-}>0\}}D_{r-}\mathrm{d}D_{r}+\int_{0}^{s}1_{\{D_{r-}>0\}}\mathrm{d}\left[D\right]{}_{r}^{c}\\
 & +\sum_{0<r\le s}\left(\max(0,D_{r})^{2}-\max(0,D_{r-})^{2}-1_{\{D_{r-}>0\}}D_{r-}\Delta D_{r}\right).
\end{split}
\label{eq:maxdelta0}
\end{equation}
Since $D$ only jumps when $b$ has a downward jump and since $\tilde{X}^{i}$
jumps to the barrier we have $-\left(\Delta b(r)\right)^{-}\le\Delta D_{r}\le0$
on the set $\left\{ D_{r-}>0\right\} $. Moreover, $D$ has bounded
paths. Since $b$ has summable downward jumps we have $\sum_{0<r\le s}1_{\{D_{r-}>0\}}\left|D_{r-}\Delta D_{r}\right|<\infty$
a.s. Hence, we can rewrite Equation \ref{eq:maxdelta0} as follows
\begin{equation}
\begin{split}\max(0,D_{s})^{2}= & \max(0,D_{0})^{2}+2\int_{0}^{s}1_{\{D_{r}>0\}}D_{r}\mathrm{d}D_{r}^{c}+\int_{0}^{s}1_{\{D_{r-}>0\}}\mathrm{d}\left[D\right]{}_{r}^{c}\\
 & +\sum_{0<r\le s}\left(\max(0,D_{r})^{2}-\max(0,D_{r-})^{2}\right).
\end{split}
\label{eq:maxdelta}
\end{equation}
Regarding the jump terms in Equation (\ref{eq:maxdelta}), assume
that there exists $r\in(0,s]$ such that $\max(0,D_{r})^{2}>\max(0,D_{r-})^{2}.$
This implies $D_{r}>0$ and $D_{r}>D_{r-}$. Since $\tilde{X}^{i}$
jumps if and only if $l^{i}$ jumps ($i=1,2)$ we obtain $\tilde{X}_{r}^{1}>\tilde{X}_{r}^{2}$
and $l_{r}^{2}-l_{r-}^{2}>l_{r}^{1}-l_{r-}^{1}$. It follows that
$l_{r}^{2}-l_{r-}^{2}>0,$ since $l^{1}$ is non-decreasing. Hence,
$l^{2}$ jumps at $r$, which implies that $\tilde{X}_{r}^{2}=b(r)$.
Thus, we obtain the contradiction $\tilde{X}_{r}^{1}>b(r).$ Therefore
we have 
\[
\sum_{0<r\le s}\left(\max(0,D_{r})^{2}-\max(0,D_{r-})^{2}\right)\le0.
\]
For the last integral in Equation (\ref{eq:maxdelta}) the Lipschitz
continuity of $\sigma$ implies 
\[
\int_{0}^{s}1_{\{D_{r}>0\}}d\langle D\rangle_{r}^{c}=\int_{0}^{s}1_{\{D_{r}>0\}}(\sigma(r,\tilde{X}_{r}^{1})-\sigma(r,\tilde{X}_{r}^{2}))^{2}\mathrm{d}r\le L^{2}\int_{0}^{s}\max(0,D_{r})^{2}\mathrm{d}r\le L^{2}\int_{0}^{s}\Gamma_{r}dr.
\]
The first integral of Equation (\ref{eq:maxdelta}) decomposes into
the following terms, which we will consider successively. By the Lipschitz
continuity of $\mu$ we have
\[
2\int_{0}^{s}1_{\{D_{r}>0\}}D_{r}(\mu(r,\tilde{X}_{r}^{1})-\mu(r,\tilde{X}_{r}^{2}))\mathrm{d}r\le2L\int_{0}^{s}\max(0,D_{r})^{2}\mathrm{d}r\le L\int_{0}^{s}\Gamma_{r}dr.
\]
Next, we have
\[
-2\int_{0}^{s}1_{\{D_{r}>0\}}D_{r}\mathrm{d}l_{r}^{1,c}\le0
\]
 and
\[
2\int_{0}^{s}1_{\{D_{r}>0\}}D_{r}\mathrm{d}l_{r}^{2,c}=2\int_{0}^{s}1_{\{\tilde{X}_{r}^{1}>b(r)\}}D_{r}\mathrm{d}l_{r}^{2,c}=0.
\]
Moreover, it follows from the Burkholder-Davis-Gundy inequality, the
Lipschitz continuity of $\sigma$ and Young's inequality that
\begin{eqnarray*}
\mathbb{E}\left[\sup_{s\le t}\int_{0}^{s}1_{\{D_{r}>0\}}D_{r}(\sigma(r,\tilde{X}_{r}^{1})-\sigma(r,\tilde{X}_{r}^{2}))\mathrm{d}W_{r}\right] & \le & C\mathbb{E}\left[\sqrt{\int_{0}^{t}1_{\{D_{r}>0\}}D_{r}^{4}\mathrm{d}r}\right]\\
 & \le & C\mathbb{E}\left[\sqrt{\Gamma_{t}\int_{0}^{t}\Gamma_{r}dr}\right]\\
 & \le & \frac{1}{2}\mathbb{E}\left[\Gamma_{t}\right]+\frac{1}{2}C^{2}\mathbb{E}\left[\int_{0}^{t}\Gamma_{r}dr\right].
\end{eqnarray*}
Putting everything together, we obtain
\[
\mathbb{E}\left[\Gamma_{t}\right]\le\Gamma_{0}+K\int_{0}^{t}\mathbb{E}\left[\Gamma_{r}\right]dr
\]
for some constant $K>0.$ Then Gronwall's lemma yields 
\begin{equation}
\mathbb{E}\left[\sup_{s\le t}\max(0,\tilde{X}_{s}^{1}-\tilde{X}_{s}^{2})^{2}\right]=\mathbb{E}[\Gamma_{t}]\le C\Gamma_{0}=C\max(0,\xi_{1}-\xi_{2})^{2}\label{eq:supmax}
\end{equation}
for some constant $C>0$. If $\xi_{1}\le\xi_{2}$ this directly yields
(\ref{enu:compprinc}). For (\ref{enu:momest2}) observe that we have
\[
\mathbb{E}\left[\sup_{s\le t}(\tilde{X}_{s}^{1}-\tilde{X}_{s}^{2})^{2}\right]\le\mathbb{E}\left[\sup_{s\le t}\max(0,\tilde{X}_{s}^{1}-\tilde{X}_{s}^{2})^{2}\right]+\mathbb{E}\left[\sup_{s\le t}\max(0,\tilde{X}_{s}^{2}-\tilde{X}_{s}^{1})^{2}\right].
\]
Then Inequality (\ref{eq:supmax}) yields $\mathbb{E}\left[\sup_{s\le t}(\tilde{X}_{s}^{1}-\tilde{X}_{s}^{2})^{2}\right]\le\tilde{C}(\xi_{1}-\xi_{2})^{2}$.
The case $p=1$ follows from Jensen's inequality. Claim (\ref{enu:comprinc2})
follows by performing similar arguments with $D=\tilde{X}^{t,\xi}-X^{t,\xi}$.

In order to prove Equation (\ref{eq:loctime}), we set
\begin{equation}
Y_{s}=Y_{s}^{t,\xi}=\int_{t}^{s}\mu(u,\tilde{X}_{u}^{t,\xi})du+\int_{t}^{r}\sigma(u,\tilde{X}_{u-}^{t,\xi})\mathrm{d}W_{u},\quad\hat{l}_{s}=\sup_{t\le r\le s}(\xi+Y_{r}-b(r))^{+}\label{eq:yappendix}
\end{equation}
and $\hat{X}=\xi+Y_{s}-\hat{l}_{s}.$ Then it is straightforward to
show that $(\hat{X},\hat{l})$ is a solution to the Skorokhod problem
associated with $Y$ and barrier $b$ (cf. \citet[Definition 2.5]{Slominski20101701}).
Since $(\tilde{X},l)$ is also a solution, we obtain Equation (\ref{eq:loctime})
by uniqueness of solutions to the Skorokhod problem (cf. \citet[Proposition 2.4]{Slominski20101701}).

Finally we prove Claim (\ref{enu:conv}). To this end let $x\le b(t)\wedge b(t-)$
and $t_{n}\nearrow t$ and $x_{n}\to x$ as $n\to\infty.$ We write
$\tilde{X}^{n}=\tilde{X}^{t_{n},x_{n}\wedge b(t_{n})}$ and $Y^{n}=Y^{t_{n},x_{n}\wedge b(t_{n})}$
(see Equation (\ref{eq:yappendix}) for the definition of $Y$). Then
we have
\begin{eqnarray*}
|\tilde{X}_{t}^{n}-x| & = & |x_{n}\wedge b(t_{n})-x+Y_{t}^{n}-\sup_{t_{n}\le r\le t}(x_{n}\wedge b(t_{n})+Y_{r}^{n}-b(r))^{+}|\\
 & \le & |x_{n}\wedge b(t_{n})-x|+\sup_{t_{n}\le r\le t}(x_{n}\wedge b(t_{n})-b(r))^{+}+2\sup_{t_{n}\le r\le t}|Y_{r}^{n}|.
\end{eqnarray*}
Squaring this inequality and taking expectations yields
\[
\mathbb{E}\left[|\tilde{X}_{t}^{n}-x|^{2}\right]\le3|x_{n}\wedge b(t_{n})-x|^{2}+3\sup_{t_{n}\le r\le t}\left((x_{n}\wedge b(t_{n})-b(r))^{+}\right)^{2}+6\mathbb{E}\left[\sup_{t_{n}\le r\le t}|Y_{r}^{n}|^{2}\right].
\]
The first two terms converge to $0$ for $n\to\infty$ since $b$
is càdlàg and $x\le b(t)\wedge b(t-)$. Regarding the last term, observe
that Jensen's and the Burkholder-Davis-Gundy inequality yields 
\begin{eqnarray*}
\mathbb{E}\left[\sup_{t_{n}\le r\le t}|Y_{r}^{n}|^{2}\right] & \le & C\int_{t_{n}}^{t}\mathbb{E}\left[\mu(s,\tilde{X}_{s}^{n})^{2}+\sigma(s,\tilde{X}_{s}^{n})^{2}\right]ds
\end{eqnarray*}
for some constant $C>0$ (not depending on $n$). It remains to prove
that $\mathbb{E}\left[\mu(s,\tilde{X}_{s}^{n})^{2}+\sigma(s,\tilde{X}_{s}^{n})^{2}\right]$
is a bounded sequence. To this end assume without loss of generality
that $\tilde{X}^{0}=\tilde{X}^{0,b(0)}$, then the linear growth of
$\mu$ and $\sigma$, the Markov property of $\tilde{X}^{0}$ and
Claim (\ref{enu:momest2}) imply 
\begin{eqnarray*}
\mathbb{E}\left[\mu(s,\tilde{X}_{s}^{n})^{2}+\sigma(s,\tilde{X}_{s}^{n})^{2}\right] & \le & C_{1}\left(1+\mathbb{E}\left[(\tilde{X}_{s}^{n}-\tilde{X}_{s}^{0})^{2}+(\tilde{X}_{s}^{0})^{2}\right]\right)\\
 & \le & C_{2}\left(1+\mathbb{E}\left[(\tilde{X}_{t_{n}}^{0}-x_{n}\wedge b(t_{n}))^{2}+(\tilde{X}_{s}^{0})^{2}\right]\right)
\end{eqnarray*}
for some $C_{1},C_{2}>0$. This is a bounded sequence by Claim (\ref{enu:momest1}),
which yields $\mathbb{E}[\sup_{t_{n}\le r\le t}|Y_{r}^{n}|^{2}]\to0$
as $n\to\infty.$
\end{proof}
\bibliographystyle{economic/ecta}
\bibliography{delstop}

\begin{thebibliography}{47}
\newcommand{\enquote}[1]{``#1''}
\expandafter\ifx\csname natexlab\endcsname\relax\def\natexlab#1{#1}\fi

\bibitem[\protect\citeauthoryear{Belomestny and Gapeev}{Belomestny and
  Gapeev}{2010}]{belomestny2010iterative}
\textsc{Belomestny, D. and P.~V. Gapeev} (2010): \enquote{An iterative
  procedure for solving integral equations related to optimal stopping
  problems,} \emph{Stochastics: An International Journal of Probability and
  Stochastics Processes}, 82, 365--380.

\bibitem[\protect\citeauthoryear{Bo and Yao}{Bo and Yao}{2007}]{bo2007strong}
\textsc{Bo, L. and R.~Yao} (2007): \enquote{{Strong Comparison Result for a
  Class of Reflected Stochastic Differential Equations with non-Lipschitzian
  Coefficients},} \emph{Frontiers of Mathematics in China}, 2, 73--85.

\bibitem[\protect\citeauthoryear{Board}{Board}{2007}]{board2007selling}
\textsc{Board, S.} (2007): \enquote{{Selling options},} \emph{Journal of
  Economic Theory}, 136, 324--340.

\bibitem[\protect\citeauthoryear{Bossy, Gobet, and Talay}{Bossy
  et~al.}{2004}]{bossy2004symmetrized}
\textsc{Bossy, M., E.~Gobet, and D.~Talay} (2004): \enquote{{Symmetrized Euler
  scheme for an efficient approximation of reflected diffusions},}
  \emph{Journal of applied probability}, 41, 877--889.

\bibitem[\protect\citeauthoryear{Carr, Jarrow, and Myneni}{Carr
  et~al.}{1992}]{carr2006alternative}
\textsc{Carr, P., R.~Jarrow, and R.~Myneni} (1992): \enquote{{Alternative
  Characterizations of American Put Options},} \emph{Mathematical Finance}, 2,
  87--106.

\bibitem[\protect\citeauthoryear{Chaleyat-Maurel, {El Karoui}, and
  Marchal}{Chaleyat-Maurel et~al.}{1980}]{chaleyat1980reflexion}
\textsc{Chaleyat-Maurel, M., N.~{El Karoui}, and B.~Marchal} (1980):
  \enquote{{R{\'e}flexion discontinue et syst{\`e}mes stochastiques},}
  \emph{The Annals of Probability}, 1049--1067.

\bibitem[\protect\citeauthoryear{De~Angelis and Kitapbayev}{De~Angelis and
  Kitapbayev}{2014}]{de2014optimal}
\textsc{De~Angelis, T. and Y.~Kitapbayev} (2014): \enquote{On the optimal
  exercise boundaries of swing put options,} \emph{arXiv preprint
  arXiv:1407.6860}.

\bibitem[\protect\citeauthoryear{Dellacherie and Meyer}{Dellacherie and
  Meyer}{1978}]{dellacherie1978probabilities}
\textsc{Dellacherie, C. and P.~A. Meyer} (1978): \emph{{Probabilities and
  potential}}, North-Holland Publishing Company.

\bibitem[\protect\citeauthoryear{Drugowitsch, Moreno-Bote, Churchland, Shadlen,
  and Pouget}{Drugowitsch et~al.}{2012}]{drugowitsch2012cost}
\textsc{Drugowitsch, J., R.~Moreno-Bote, A.~K. Churchland, M.~N. Shadlen, and
  A.~Pouget} (2012): \enquote{{The cost of accumulating evidence in perceptual
  decision making},} \emph{The Journal of Neuroscience}, 32, 3612--3628.

\bibitem[\protect\citeauthoryear{Du~Toit and Peskir}{Du~Toit and
  Peskir}{2007}]{du2007trap}
\textsc{Du~Toit, J. and G.~Peskir} (2007): \enquote{The trap of complacency in
  predicting the maximum,} \emph{The Annals of Probability}, 340--365.

\bibitem[\protect\citeauthoryear{Ekstroem and Vaicenavicius}{Ekstroem and
  Vaicenavicius}{2016}]{ekstrom2016optimal}
\textsc{Ekstroem, E. and J.~Vaicenavicius} (2016): \enquote{Optimal liquidation
  of an asset under drift uncertainty,} \emph{SIAM Journal on Financial
  Mathematics}, 7, 357--381.

\bibitem[\protect\citeauthoryear{Fudenberg, Strack, and Strzalecki}{Fudenberg
  et~al.}{2015}]{fudenberg2015stochastic}
\textsc{Fudenberg, D., P.~Strack, and T.~Strzalecki} (2015):
  \enquote{{Stochastic Choice and Optimal Sequential Sampling},} \emph{arXiv
  preprint arXiv:1505.03342}.

\bibitem[\protect\citeauthoryear{Gapeev and Peskir}{Gapeev and
  Peskir}{2004}]{gapeev2004wiener}
\textsc{Gapeev, P.~V. and G.~Peskir} (2004): \enquote{The Wiener sequential
  testing problem with finite horizon,} \emph{Stochastics and Stochastic
  Reports}, 76, 59--75.

\bibitem[\protect\citeauthoryear{Gapeev and Peskir}{Gapeev and
  Peskir}{2006}]{gapeev2006wiener}
---\hspace{-.1pt}---\hspace{-.1pt}--- (2006): \enquote{The Wiener disorder
  problem with finite horizon,} \emph{Stochastic processes and their
  applications}, 116, 1770--1791.

\bibitem[\protect\citeauthoryear{Gapeev and Shiryaev}{Gapeev and
  Shiryaev}{2013}]{gapeev2013}
\textsc{Gapeev, P.~V. and A.~N. Shiryaev} (2013): \enquote{Bayesian Quickest
  Detection Problems for Some Diffusion Processes,} \emph{Advances in Applied
  Probability}, 45, 164--185.

\bibitem[\protect\citeauthoryear{Gapeev et~al.}{Gapeev
  et~al.}{2006}]{gapeev2006discounted}
\textsc{Gapeev, P.~V. et~al.} (2006): \enquote{Discounted optimal stopping for
  maxima in diffusion models with finite horizon,} \emph{Electron. J. Probab},
  11.

\bibitem[\protect\citeauthoryear{Glover, Hulley, Peskir et~al.}{Glover
  et~al.}{2013}]{glover2013three}
\textsc{Glover, K., H.~Hulley, G.~Peskir, et~al.} (2013):
  \enquote{Three-dimensional Brownian motion and the golden ratio rule,}
  \emph{The Annals of Applied Probability}, 23, 895--922.

\bibitem[\protect\citeauthoryear{Hopenhayn and Nicolini}{Hopenhayn and
  Nicolini}{1997}]{hopenhayn1997optimal}
\textsc{Hopenhayn, H. and J.~P. Nicolini} (1997): \enquote{{Optimal
  unemployment insurance},} \emph{Journal of Political Economy}, 105, 412--438.

\bibitem[\protect\citeauthoryear{Jacka}{Jacka}{1991}]{jacka1991optimal}
\textsc{Jacka, S.} (1991): \enquote{{Optimal Stopping and the American Put},}
  \emph{Mathematical Finance}, 1, 1--14.

\bibitem[\protect\citeauthoryear{Jacka and Lynn}{Jacka and
  Lynn}{1992}]{jacka1992finite}
\textsc{Jacka, S. and J.~Lynn} (1992): \enquote{{Finite-Horizon Optimal
  Stopping, Obstacle Problems and the Shape of the Continuation Region},}
  \emph{Stochastics: An International Journal of Probability and Stochastic
  Processes}, 39, 25--42.

\bibitem[\protect\citeauthoryear{Karatzas and Shreve}{Karatzas and
  Shreve}{1991}]{karatzas1991brownian}
\textsc{Karatzas, I.~A. and S.~E. Shreve} (1991): \emph{{Brownian motion and
  stochastic calculus}}, vol. 113, Springer.

\bibitem[\protect\citeauthoryear{Kim}{Kim}{1990}]{kim1990analytic}
\textsc{Kim, I.} (1990): \enquote{{The Analytic Valuation of American
  Options},} \emph{Review of Financial Studies}, 3, 547--572.

\bibitem[\protect\citeauthoryear{Kotlow}{Kotlow}{1973}]{kotlow1973}
\textsc{Kotlow, D.~B.} (1973): \enquote{{A Free Boundary Problem Connected with
  the Optimal Stopping Problem for Diffusion Processes},} \emph{Transactions of
  the American Mathematical Society}, 184, pp. 457--478.

\bibitem[\protect\citeauthoryear{Kruse and Strack}{Kruse and
  Strack}{2015}]{krusestrack2013optimal}
\textsc{Kruse, T. and P.~Strack} (2015): \enquote{{Optimal Stopping with
  Private Information},} \emph{To appear in the Journal of Economic Theory}.

\bibitem[\protect\citeauthoryear{Kunita}{Kunita}{2004}]{kunita2004stochastic}
\textsc{Kunita, H.} (2004): \enquote{{Stochastic differential equations based
  on L{\'e}vy processes and stochastic flows of diffeomorphisms},} in
  \emph{{Real and stochastic analysis}}, Springer, 305--373.

\bibitem[\protect\citeauthoryear{Liptser and Shiryaev}{Liptser and
  Shiryaev}{2013}]{liptser2013statistics}
\textsc{Liptser, R. and A.~N. Shiryaev} (2013): \emph{Statistics of random
  Processes: I. general Theory}, vol.~5, Springer Science \& Business Media.

\bibitem[\protect\citeauthoryear{McCall}{McCall}{1970}]{mccall1970economics}
\textsc{McCall, J.~J.} (1970): \enquote{{Economics of information and job
  search},} \emph{The Quarterly Journal of Economics}, 84, 113--126.

\bibitem[\protect\citeauthoryear{Milgrom and Segal}{Milgrom and
  Segal}{2002}]{milgrom2002envelope}
\textsc{Milgrom, P. and I.~Segal} (2002): \enquote{{Envelope theorems for
  arbitrary choice sets},} \emph{Econometrica}, 70, 583--601.

\bibitem[\protect\citeauthoryear{M{\"u}ller and Siegmund}{M{\"u}ller and
  Siegmund}{1994}]{muller1994change}
\textsc{M{\"u}ller, H.-G. and D.~Siegmund} (1994): \enquote{Change-point
  problems,} Ims.

\bibitem[\protect\citeauthoryear{{\"O}nskog and Nystr{\"o}m}{{\"O}nskog and
  Nystr{\"o}m}{2010}]{onskog2010weak}
\textsc{{\"O}nskog, T. and K.~Nystr{\"o}m} (2010): \enquote{{Weak approximation
  of obliquely reflected diffusions in time-dependent domains},} \emph{Journal
  of Computational Mathematics}, 28, 579--605.

\bibitem[\protect\citeauthoryear{Peskir}{Peskir}{2005{\natexlab{a}}}]{peskir2005change}
\textsc{Peskir, G.} (2005{\natexlab{a}}): \enquote{A change-of-variable formula
  with local time on curves,} \emph{Journal of Theoretical Probability}, 18,
  499--535.

\bibitem[\protect\citeauthoryear{Peskir}{Peskir}{2005{\natexlab{b}}}]{peskir2005american}
---\hspace{-.1pt}---\hspace{-.1pt}--- (2005{\natexlab{b}}): \enquote{{On the
  American Option Problem},} \emph{Mathematical Finance}, 15, 169--181.

\bibitem[\protect\citeauthoryear{Peskir}{Peskir}{2005{\natexlab{c}}}]{peskir2005russian}
---\hspace{-.1pt}---\hspace{-.1pt}--- (2005{\natexlab{c}}): \enquote{{The
  Russian Option: Finite Horizon},} \emph{Finance and Stochastics}, 9,
  251--267.

\bibitem[\protect\citeauthoryear{Peskir}{Peskir}{2007}]{peskir2007change}
---\hspace{-.1pt}---\hspace{-.1pt}--- (2007): \enquote{A change-of-variable
  formula with local time on surfaces,} in \emph{S{\'e}minaire de
  probabilit{\'e}s XL}, Springer, 70--96.

\bibitem[\protect\citeauthoryear{Peskir}{Peskir}{2014}]{peskir2014quickest}
---\hspace{-.1pt}---\hspace{-.1pt}--- (2014): \enquote{Quickest detection of a
  hidden target and extremal surfaces,} \emph{The Annals of Applied
  Probability}, 24, 2340--2370.

\bibitem[\protect\citeauthoryear{Peskir and Samee}{Peskir and
  Samee}{2011}]{peskir2011british}
\textsc{Peskir, G. and F.~Samee} (2011): \enquote{The British put option,}
  \emph{Applied Mathematical Finance}, 18, 537--563.

\bibitem[\protect\citeauthoryear{Peskir and Samee}{Peskir and
  Samee}{2013}]{peskir2013british}
---\hspace{-.1pt}---\hspace{-.1pt}--- (2013): \enquote{The British call
  option,} \emph{Quantitative Finance}, 13, 95--109.

\bibitem[\protect\citeauthoryear{Peskir and Shiryaev}{Peskir and
  Shiryaev}{2006{\natexlab{a}}}]{pevskir2006optimal}
\textsc{Peskir, G. and A.~Shiryaev} (2006{\natexlab{a}}): \emph{{Optimal
  Stopping and Free-Boundary Problems}}, vol.~10, Birkhauser.

\bibitem[\protect\citeauthoryear{Peskir and Shiryaev}{Peskir and
  Shiryaev}{2006{\natexlab{b}}}]{peskir2006optimal}
---\hspace{-.1pt}---\hspace{-.1pt}--- (2006{\natexlab{b}}): \emph{{Optimal
  Stopping and Free-Boundary Problems}}, vol.~10, Birkhauser.

\bibitem[\protect\citeauthoryear{Protter}{Protter}{2005}]{protter2005stochastic}
\textsc{Protter, P.} (2005): \emph{{Stochastic Integration and Differential
  Equations}}, {Applications of Mathematics}, Springer.

\bibitem[\protect\citeauthoryear{Rutkowski}{Rutkowski}{1980}]{rutkowski1980stochastic}
\textsc{Rutkowski, M.} (1980): \enquote{{Stochastic integral equations with a
  reflecting barrier},} \emph{Demonstratio Math. XIII (2)}, 483--507.

\bibitem[\protect\citeauthoryear{Saisho}{Saisho}{1987}]{saisho1987stochastic}
\textsc{Saisho, Y.} (1987): \enquote{{Stochastic differential equations for
  multi-dimensional domain with reflecting boundary},} \emph{Probability Theory
  and Related Fields}, 74, 455--477.

\bibitem[\protect\citeauthoryear{Shiryaev}{Shiryaev}{2007}]{shiryaev2007optimal}
\textsc{Shiryaev, A.~N.} (2007): \emph{Optimal stopping rules}, vol.~8,
  Springer Science \& Business Media.

\bibitem[\protect\citeauthoryear{Slominski and Wojciechowski}{Slominski and
  Wojciechowski}{2010}]{Slominski20101701}
\textsc{Slominski, L. and T.~Wojciechowski} (2010): \enquote{{Stochastic
  Differential Equations with Jump Reflection at Time-Dependent Barriers},}
  \emph{Stochastic Processes and their Applications}, 120, 1701--1721.

\bibitem[\protect\citeauthoryear{Villeneuve}{Villeneuve}{2007}]{villeneuve2007threshold}
\textsc{Villeneuve, S.} (2007): \enquote{{On threshold strategies and the
  smooth-fit principle for optimal stopping problems},} \emph{Journal of
  applied probability}, 44, 181--198.

\bibitem[\protect\citeauthoryear{Zhitlukhin and Muravlev}{Zhitlukhin and
  Muravlev}{2013}]{zhitlukhin2013chernoff}
\textsc{Zhitlukhin, M. and A.~Muravlev} (2013): \enquote{On Chernoff's
  Hypotheses Testing Problem for the Drift of a Brownian Motion,} \emph{Theory
  of Probability \& Its Applications}, 57, 708--717.

\bibitem[\protect\citeauthoryear{Zhitlukhin and Shiryaev}{Zhitlukhin and
  Shiryaev}{2013}]{zhitlukhin2013bayesian}
\textsc{Zhitlukhin, M. and A.~Shiryaev} (2013): \enquote{Bayesian disorder
  problems on filtered probability spaces,} \emph{Theory of Probability \& Its
  Applications}, 57, 497--511.

\end{thebibliography}

\end{document}